\documentclass[a4paper,12pt]{article}

\usepackage[latin1]{inputenc}
\usepackage{amsfonts}
\usepackage{amsthm}
\usepackage{amsmath}
\usepackage{amsfonts}
\usepackage{latexsym}
\usepackage{amssymb}
\usepackage{dsfont}

\newtheorem{fed}{Definition}[section]
\newtheorem{teo}[fed]{Theorem}
\newtheorem*{teo*}{Theorem}
\newtheorem{lem}[fed]{Lemma}
\newtheorem{cor}[fed]{Corollary}
\newtheorem{pro}[fed]{Proposition}
\theoremstyle{definition}
\newtheorem{rem}[fed]{Remark}
\newtheorem{conj}[fed]{Conjecture}

\newtheorem{exa}[fed]{Example}

\newtheorem{num}[fed]{}

\newcommand{\IN}[1]{\mathbb {I} _{#1}}
\def\In{\mathbb {I} _n}
\def\CC{\mathbb {C}}

\def\RR{\mathbb {R}}

\def\IM{\mathbb {I} _m}

\def\suml{\sum\limits}
\def\QEDP{\tag*{\QED}}

\def\ov{\overline}

\def\bce{\begin{center}}
\def\ece{\end{center}}
\newcommand{\trivial}{\{0\}}
 
\DeclareMathOperator{\FP}{RSP\,}

\def\gr{\text{\rm Gr }}
\def\cO{{\mathcal O}} 
\def\cD{\mathcal D}
\def\t{{\mathbf t}} 
\def\FS{\cN_w  = (w_i\, ,\, \cN_i)_{i\in \IM}}
\def\fori{for every $i \in \IM\,$}

\def\subim{_{i\in \IN{m}}\,}

\def\rk{\text{\rm rk}}
\def\noi{\noindent}
\def\cF{\mathcal F}

\def\QED{\hfill $\square$}
\def\EOE{\hfill $\triangle$}
\renewcommand{\qed}{\hfill $\square$}
\def\bdem{\begin{proof}}
\def\edem{\end{proof}}
\newcommand{\peso}[1]{ \quad \text{ #1 } \quad }

\def\uno{\mathds{1}}
\def\bm{\left[\begin{array}}
\def\em{\end{array}\right]}

\def\ben{\begin{enumerate}}
\def\een{\end{enumerate}}
\def\bit{\begin{itemize}}
\def\eit{\end{itemize}}
\def\barr{\begin{array}}
\def\earr{\end{array}}
\def\igdef{\ \stackrel{\mbox{\tiny{def}}}{=}\ }

\def\k{\mathbf{k}}

\def\v{\mathbf{v}}
\def\x{\mathbf{x}}

\def\eps{\varepsilon}

\def\la{\lambda}
\def\al{\alpha}

\def\N{\mathbb{N}}
\def\R{\mathbb{R}}
\def\C{\mathbb{C}}

\def\cA{\mathcal{A}}

\def\cI{\mathcal{I}}

\def\cE{\mathcal{E}}

\def\cH{\mathcal{H}}
\def\cK{\mathcal{K}}

\def\cP{\mathcal{P}}
\def\cQ{\mathcal{Q}}
\def\cR{{\cal R}}
\def\cS{{\cal S}}
\def\cT{{\cal T}}
\def\cM{{\cal M}}
\def\cN{{\cal N}}
\def\cV{{\cal V}}
\def\cU{{\cal U}}
\def\cW{{\cal W}}

\def\ese{\mathcal{S}}
\def\ete{\mathcal{T}}
\def\eme{\mathcal{M}}
\def\ene{\mathcal{N}}

\def\vacio{\emptyset}
\def\orto{^\perp}
\def\inc{\subseteq}
\def\Inc{\supseteq}
\def\sii{ if and only if }
\def\inv{^{-1}}
\def\*A{\#\sb A}

\def\rai{^{1/2}}

\def\api{\langle}
\def\cpi{\rangle}
\def\inv{^{-1}}

\DeclareMathOperator{\Preal}{Re} 
 \DeclareMathOperator{\tr}{tr}

\DeclareMathOperator{\gen}{span}

\def\SP{L(m,\k,d)}
\def\RS{\cR \cS }
\def\RSv{\cR \cS_\v \,}
\def\OPv{\cO\cP_\v}
\def\LOPv{\Lambda(\OPv )}

\def\PRO{\cP \cR\cS }
\def\RSO{\mbox{\rm RSO}} 
\def\PROv{\cP \cR\cS _\v\, }

\def\DRSv{\cD\cP_\v\, }
\def\RSV{\cV= \{V_i\}_{i\in \, \IN{m}}}
\def\RSU{\cU= \{U_i\}_{i\in \, \IN{m}}}
\def\RSW{\cW= \{W_i\}_{i\in \, \IN{m}}}

\def\RSF{\cF= \{f_i\}_{i\in \, \IN{m}}}

\newcommand{\hil}{\mathcal{H}}
\newcommand{\op}{L(\mathcal{H})}
\newcommand{\lhk}{L(\mathcal{H} , \mathcal{K})}
\newcommand{\lkh}{L(\mathcal{K} , \mathcal{H})}

\newcommand{\posop}{L(\mathcal{H})^+}

\def\H{{\cal H}}
\def\lh{{L(\H)}}
\def\lh+{{\lh^+}}

\def\glh{\mathcal{G}\textit{l}\,(\cH)}
\def\glk{\mathcal{G}\textit{l}\,(\cK)}

\newcommand{\cene}{\mathbb{C}^n}
\newcommand{\mat}{\mathcal{M}_d(\mathbb{C})}
\newcommand{\matn}{\mathcal{M}_n(\mathbb{C})}

\newcommand{\matsa}{\mathcal{H}(n)}

\newcommand{\matu}{\mathcal{U}(n)}

\newcommand{\matpos}{\mat^+}
\newcommand{\matposn}{\matn^+}
\newcommand{\matnpos}{\matn^+}
\newcommand{\matinv}{\mathcal{G}\textit{l}\,(n)}
\newcommand{\matinvd}{\mathcal{G}\textit{l}\,(d)}
\def\gld{\matinvd^+}

\newcommand{\matrec}[1]{\mathcal{M}_{#1} (\mathbb{C})}

\def\beq{\begin{equation}}
\def\eeq{\end{equation}}

\def\pausa{\medskip\noi}

\begin{document}

\title{{\bf Duality in  reconstruction systems}
}
\author{Pedro G. Massey, Mariano A. Ruiz  and Demetrio Stojanoff\thanks{Partially supported by CONICET 
(PIP 5272/05) and  Universidad de La PLata (UNLP 11 X472).} }
\author{P. G. Massey, M. A. Ruiz and D. Stojanoff \\ {\small Depto. de Matem\'atica, FCE-UNLP,  La Plata, Argentina
and IAM-CONICET  
}}
\date{}

\maketitle

\begin{abstract} We consider the notion of finite dimensional reconstructions 
systems (RS's), which includes  the fusion frames as projective RS's. 
We study  erasures, some geometrical properties of these spaces, 
the spectral picture of the set of all dual systems of a fixed RS, 
the spectral picture of the set of RS operators for the projective systems
with fixed weights and  the structure
of the minimizers of the joint potential in this setting.  We give several examples. 
\end{abstract}

\tableofcontents

\def\coma{\, , \, }

\bigskip

\section{Introduction} 
In this paper we study the notion of finite dimensional reconstruction systems,
which gives a new framework for fusion and vector frames. 
Fusion frames (briefly FF's) were introduced 
under the name of ``frame of subspaces'' in \cite{[CasKu]}. They arise naturally as a generalization of the usual frames of vectors for a Hilbert space $\hil$.
Several applications of FF's have been studied, for example, 
sensor networks \cite{[CasKuLiRo]}, neurology \cite{RJ}, 
coding theory \cite{Bod}, \cite {Pau}  , \cite{GK}, among others. 
We refer the reader to 
\cite{[CasKuLi]} and the references therein  
for a detailed treatment of the FF theory. Further developments 
can be found in  \cite{CF}, \cite {[CasKu2]} and \cite{RS}. 

Given $m \in \N$ we denote by $\IM = \{1, \dots , m\} \inc \N$. 
In the finite dimensional setting, a FF is a sequence 
$\FS$ where each $w_i \in \R_{>0}$ and  the $\ene_i\inc \C^d$ are subspaces that 
generate $\C^d$. 
The synthesis operator of  $\ene_w$ is usually defined as 
$$
\barr{rl}
T_{\ene_w} :\cK_{\ene_w} \igdef \bigoplus\subim \ene_i \to 
\C^d & \peso{given by} T_{\ene_w} (x_i)\subim = \sum \subim w_i \, x_i \ .
\earr
$$ 
Its adjoint, the so-called analysis operator of $\cN_w\,$, is given  
by $T_{\ene_w}^* y = (w_i \, P_{\ene_i} \, y)\subim$ for 
$y\in \C^d$, where $P_{\ene_i}$ denotes the orthogonal projection onto $\ene_i\,$. 
The frame $\ene_w$ induces a linear 
encoding-decoding scheme that can be described in terms of these operators.

The previous setting for the theory of FF's presents some technical difficulties. For example the domain of $T_{\ene_w}$ 
relies strongly on the subspaces of the fusion frame.  
In particular, any change on the subspaces modifies the domain 
of the operators preventing smooth perturbations of these objects.
Moreover, this kind of rigidity on the definitions implies that the 
notion of a dual FF is not clear. 

An alternative approach to the fusion frame (FF) theory comes from the theory of protocols introduced in \cite{Bod} and the theory of reconstruction systems considered in \cite {MRS} and \cite{P}.
In this context, we fix the dimensions 
$\dim \ene_i = k_i$ and consider a universal space 
$$\cK = \cK_{m\coma \k} \igdef  \bigoplus_{i\in \, \IN{m}}  \ \C^{k_i} \ , 
\peso{where} \k = (k_1 \coma \dots \coma k_m)   \in \N^m \ .
$$
A reconstruction system (RS) is a sequence  
$\RSV$ such that $V_i \in L(\C^d\coma \C^{k_i})$ \fori, 
which allows the construction of an encoding-decoding algorithm 
(see Definition \ref {defi recons} for details). We denote by 
$\RS= \RS(m,\k,d)$ the set of all RS's with these fixed parameters. 
Observe that, if  $\FS$ is a FF, it can be modeled as a system 
$\RSV \in \RS$ such that 
$V_i^*V_i = w_i^2 \, P_{\mathcal N_i}\,$ \fori. 
These systems are called projective RS's. 

On the other hand, a general RS arise from 
a usual vector frame by grouping together the elements of the frame. Thus, the 
coefficients involved in the 
encoding-decoding scheme of RS are vector valued, and they lie in the space $\cK$.

The main advantage of the RS framework with respect to the fusion 
frame formalism is that each (projective) RS  has many RS's that 
are {\it dual} systems. In particular, the canonical dual RS
remains being a RS (for details and definitions see Section \ref{basic}). 
In contrast, it is easy to give examples of a FF
such that its canonical dual is not a fusion frame. 
There exists a notion of duality among fusion frames defined by Gavruta 
(see \cite{Gav}), where the reconstruction formula
of a fixed $\cV$ involves the FF operator $S_\cV$ of $\cV$. 
Nevertheless,  
in the context of RS's, we show that the notion of dual systems  
can be described and characterized in a quite natural way. 
On the other hand, the RS framework (see Section 2 for a detailed description) 
allows to make not only a metric but 
also a differential geometric study of the set of RS's, which will be developed in 
Section 4 of this paper.

Let us fix the parameters $(m,\k,d)$ and the sequence 
$\v=(v_i)_{i\in \IM}\in \R_{>0}^m\,$ of weights. 
In this work we study some properties of 
the sets $\RS = \RS(m,\k,d)$ of RS's and $\PROv = \PROv(m,\k,d)$ of projective 
systems with fixed weights $\v$.  
In section 3 we study erasures in this context. We  show conditions  
which guarantee that, after erasing some of its components, 
the system keeps being a RS, and we exhibit adequate bounds for it. 
In section 4 we present  a geometrical description of both sets 
$\RS $ and $\PROv$, and give a sufficient  condition 
(the notion of irreducible systems) 
in order that the operation of taking RS operators 
$\PROv\ni \cV \mapsto S_\cV$ (see Definition \ref {defi recons}) has smooth local 
cross sections. In section 5 we study the spectral picture of the set $\cD(\cV)$ of all dual systems
for a fixed $\cV \in \RS$, and the set $\OPv$ of the RS operators of all systems in $\PROv$. 

Finally, in section 6 we focus on the main problem of the paper, which needs 
the results of the previous sections:  
Let $\DRSv  \igdef \big\{ \, (\cV,\,\cW) \in  \PROv \times \RS \ : \ 
\cW\in \cD(\cV) \, \big\}$.
We look for  pairs $(\cV \coma \cW) $ 
which have the best minimality properties. If there exist tight systems in $\PROv\,$
(systems whose RS operator is a multiple of the identity) then the pair $(\cV \coma \cV^\#)$
is minimal, where $\cV^\#$ is the canonical dual of $\cV$ (see Defintion \ref{CanDua}). 
Nevertheless, this is not always the case (see \cite {comp} or \cite {MRS}). 
Therefore we  define a joint RS potential given by 
$\DRSv \ni (\cV \coma \cW) \mapsto \FP(\cV \coma \cW) 
= \tr \, S_\cV^2 + \tr \, S_\cW^2 
\in \R_{>0}\,$, which is similar to the potential used in \cite {Phys} for 
 vector frames. 
The minimizers of  $\FP$ are 
those pairs which are the best analogue of a tight pair. 
The main results in this direction are that: 
\bit
\item There exist $\la_\v = \la_\v(m,\k,d) \in \R_{>0}^d$ such that a pair 
$(\cV \coma \cW) \in \DRSv$ is a minimizer 
for the $\FP$  \sii $\cW = \cV^\#$ and the vector of eigenvalues 
$\la(S_\cV) = \la_\v\,$. 
\item Every such  $\cV$ can be decomposed as a orthogonal sum 
of tight projective RS's, where the quantity of components and their tight constants 
are the same for every minimizer. 
\eit
In section 7 we give some examples of these problems, 
showing sets of parameters for which the vector $\la_\v$ and
all minimizers $\cV \in \PROv$ can be explicitly computed. We also present
a conjecture which suggest an easy way to compute the vector $\la_\v\,$, as 
the minimal element in the spectral picture of $\OPv$ with respect 
to the majorization (see Conjecture \ref{la conj}).

\subsection*{General notations.}
Given $m \in \N$ we denote by $\IM = \{1, \dots , m\} \inc \N$ and 
$\uno = \uno_m  \in \R^m$ denotes the vector with all its entries equal to $1$. 
For a vector $x\in \R^m$ we denote by $x^\downarrow$ the rearrangement
of $x$ in a decreasing order, and $(\R^m)^\downarrow = \{ x\in \R^m : x = x^\downarrow\}$
the set of ordered vectors. 

\pausa 
Given $\cH \cong \C^d$  and $\cK \cong \C^n$, we denote by $\lhk $ 
the space of linear operators $T : \cH \to \cK$. 
Given an operator $T \in \lhk$, $R(T) \inc \cK$ denotes the
image of $T$, $\ker T\inc \cH$ the null space of $T$ and $T^*\in \lkh$ 
the adjoint of $T$. If $d\le n$ we say that $U\in \lhk$ is an isometry 
if $U^*U = I_\cH\,$. In this case, $U^*$ is called a coisometry. 
If $\cK = \cH$ we denote by $\op = L(\cH \coma \cH)$, 
by $\glh$ the group of all invertible operators in $\op$, 
 by $\posop $ the cone of positive operators and by
$\glh^+ = \glh \cap \posop$. 
If $T\in \op$, we  denote by   
$\sigma (T)$ the spectrum of $T$, by rk $T $  the rank of $T$,
and by $\tr T$ the trace of $T$. By fixing orthonormal basis (onb) 
of the Hilbert spaces involved, we shall identify operators with 
matrices, using the following notations:

\pausa 
By $\matrec{n,d} \cong L(\C^d \coma \C^n)$ we denote the space of complex $n\times d$ matrices. 
If $n=d$ we write $\matn = \matrec{n,n}$.  
$\matsa$ is the $\R$-subspace of selfadjoint matrices,  
$\matinv$ the group of all invertible elements of $\matn$, $\matu$ the group 
of unitary matrices, 
$\matposn$ the set of positive semidefinite
matrices, and $\matinv^+ = \matposn \cap \matinv$. 
If $d\le n$, we denote by $\cI(d\coma n) \inc 
\matrec{n\coma d}$ the set of isometries, i.e. those 
$U\in \matrec{n\coma d} $ such that $U^*U = I_d\,$.

\pausa
If $W\inc \cH$ is a subspace we denote by $P_W \in \posop$ the orthogonal 
projection onto $W$, i.e. $R(P_W) = W$ and $\ker \, P_W = W^\perp$. 
For vectors on $\cene$ we shall use the euclidean norm.  
On the other hand,  for matrices $T\in \matn$ we shall use both 
\begin{enumerate} 
\item The spectral norm $\|T\| = \|T\|_{sp}= \max\limits_{\|x\|=1}\|Tx\|$. 
\item The Frobenius norm $\|T\|_{_2} = (\tr \, T^*T )\rai = 
\big( \, \suml_{i,j \in \In } |T_{ij}|^2 \, \big)\rai$. This norm  is induced by the inner product
$\api A,\ B\cpi=  \tr \, B^*A \,$,  for $A, B \in \matn$. 

\end{enumerate}

\section{Basic framework of reconstruction systems}\label{basic}

In what follows we consider $(m,\k,d)$-reconstruction systems, which
are more general linear systems than those considered in
\cite{BF}, \cite{Bod}, \cite{Pau}, \cite{BodPau},  \cite{HolPau} and
\cite{MR}, that also have an associated reconstruction algorithm.  

\begin{fed}\label{defi recons}\rm
Let $m, d \in \N$ and $\k = (k_1 \coma \dots \coma k_m)   \in \N^m $. 
\ben 
\item We shall abbreviate the above description by saying that  $(m, \k , d)$
is a set of parameters. We denote by $n = \tr \, \k \igdef \sum_{i\in \, \IN{m}} k_i$
and assume that $n\ge d$. 
\item We denote by $
\cK = \cK_{m\coma \k} \igdef  \bigoplus_{i\in \, \IN{m}}  \ \C^{k_i}\cong \C^n $.
We shall often write each direct summand by $\cK_i = \C^{k_i}\,$. 

\item Given a space $\cH \cong \C^d$ we denote by 
$$\SP \igdef \bigoplus_{i\in \IM} L(\cH \coma \cK_i) \cong \lhk 
\cong \bigoplus_{i\in \, \IN{m}}  \ \matrec{k_i\coma d}\cong \matrec{n,d}\ .
$$ 
A typical  element 
of $\SP$ is a system $\RSV $ such that 
each $V_i \in L(\cH \coma \cK_i)$. 
\item
A family $\cV= \{V_i\}_{i\in \, \IN{m}}\in \SP$  is an 
$(m,\k,d)$-reconstruction system (RS)
for $\cH$ if
\beq\label{S de RS}
S_\cV \igdef \sum\limits_{i\in \, \IN{m}} V_i^*V_i  \in \glh^+\, ,
\eeq
i.e., if $S_\cV$ is  invertible. $S_\cV$ is called the {\bf RS operator} of $\cV$. In this case, 
the $m$-tuple  $\k = (k_1 \coma \dots \coma k_m)   \in \N^m $ satisfies that 
$n= \tr \k  \ge d$. 

We shall denote by $\RS = \RS(m,\k,d)$ the set of all $(m,\k,d)$-RS's for $\cH\cong \C^d$. 

\item
The system $\cV $  is said to be {\bf projective} if 
there exists a sequence  $\v= (v_i)_{i \in \IN{m}}\in \R_{+}^m $ 
of positive numbers, the 
weights of $\cV$, such that 
$$
V_i \, V_i^* = v_i^2 \, I_{\cK_i} \ , \peso{for every} i \in \IN{m} \ . 
$$
In this case, the following properties hold: 
\ben
\item The weights can be computed directly, since each  $v_i = \|V_i\|_{sp}\ $. 
\item Each $V_i = v_i U_i$ for a  coisometry $U_i \in L(\cH \coma \cK_i)$. 
Thus $V_i^*V_i=v_i^2\, 
P_{R(V_i^*)} \in \posop$ \fori.  

\item 
$S_\cV = \sum_{i\in \, \IN{m}} v_i^2 \,P_{R(V_i^*)}$ as in fusion frame theory. 
\een
We shall denote by  $\PRO = \PRO(m,\k,d)$ the set of all projective elements of $\RS$.

\item
The {\bf analysis} operator of the system $\cV$ is defined by 
$$
T_\cV \ : \ \cH\rightarrow
\cK  =
 \bigoplus_{i\in \, \IN{m}}  \cK_i\ \ \text{given by } \ \ T_\cV\, x= (V_1 \, x \coma \dots \coma V_m\, x)  
\ , \peso{for} x\in \cH \ .
$$ 
\item 
Its adjoint  $T_\cV^*$ is called the {\bf synthesis} operator of the system $\cV$, and it satisfies that
$$
T_\cV^* : \cK = \bigoplus_{i\in \, \IN{m}}  \cK_i \rightarrow \cH  \peso{is given by}
T_\cV ^* \, \big(\,  (y_i)_{i\in \, \IN{m}}\big) =\sum_{i\in \, \IN{m}} V_i^*\, y_i \ . 
$$
Using the previous notations and definitions we have that $S_\cV = T_\cV^*\  T_\cV\, $. 
\item The frame constants in this context are the following: $\cV$ is a RS \sii
\beq\label{const RS}
A_\cV \, \|x\|^2 \, \le \, \api S_\cV \, x\coma x\cpi \, = \sum\subim \, \|V_i \, x \|^2  \,
 \,  \le \, B_\cV \, \|x\|^2 
\eeq
for every $x\in \cH$, 	
where $0< A_\cV   =\la_{\min} (S_\cV) = \|S_\cV\inv \| \inv\le 
\la_{\max} (S_\cV) = \|S_\cV \| = B_\cV \,$. 
\item As usual, we say that 
$\cV$ is {\bf tight} if $A_\cV = B_\cV\,$. In other words, the system 
$\cV \in \RS(m,\k,d)$ is 
tight if and only if $S_\cV = \frac{\tau}{d}  \, I_\H\,$, where $\tau = \sum\subim v_i^2 \, k_i\,$. 
\item The Gram matrix of $\cV$ is  $G_\cV = T_\cV\, T_\cV^* \in 
L(\cK)^+ \cong  \matnpos$,  where the size of $G_\cV$ viewed 
as a matrix is $n = \tr \, \k = \sum\subim \, k_i = \dim \cK$. 
\item Given $U \in \matinvd$, we define $\cV\cdot U  \igdef 
\{V_i\,U\}_{i\in \, \IN{m}} \in \RS(m,\k,d)$.  \EOE
\een
\end{fed}

\begin{rem}
Let $\RSV \in \RS$ such that every $V_i \neq 0$.  In case that $\k = \uno_m\,$, 
then $\cV$ can be identified with a vector frame, since each $V_i : \C^d \to \C$ is in fact a 
vector $0\neq f_i \in \C^d$. In the same manner, the projective RS's can be seen as fusion frames. 
Here the identification is given by 
$V_i \simeq \big(\,\|V_i\| \coma R(V_i^*)\,\big)$ \fori. 
\EOE

\end{rem}

\begin{fed}\label{CanDua} \rm
For every $\RSV\in \RS(m,\k,d)$, we define the system 
$$
\cV^\# \igdef \cV\cdot S_\cV ^{-1} =  \{V_i\,S_\cV ^{-1}\}_{i\in \, \IN{m}} \in \RS(m,\k,d) \ , 
$$
called  the {\bf canonical dual} RS associated to $\cV$.  
\EOE
\end{fed}

\begin{rem} 
Given $\RSV \in \RS$ with 
 $S_\cV = \sum_{i\in \, \IN{m}} V_i^*V_i\,$,  then 
\begin{equation}\label{hecho1}
\sum_{i\in \, \IN{m}}S_\cV\, ^{-1}\, V_i^*V_i=I_\cH  \ ,
\ \ \text{ and } \ \ \sum_{i\in \, \IN{m}} V_i^*V_i\ S_\cV\, ^{-1}=I_ \cH \ . 
\end{equation} 
Therefore, we obtain the reconstruction formulas
\begin{equation}\label{ec recons}
x= \sum_{i\in \, \IN{m}} S_\cV^{-1} \, V_i^* \, (V_i \,x) 
= \sum_{i\in \, \IN{m}} V_i^* \,V_i(S_\cV^{-1}\, x) \peso{for every} x\in \cH \ .
\end{equation} 
Observe that, by Eq. \eqref{hecho1}, we see that 
the canonical dual $\cV^\#$ satisfies that 
\beq\label{SMP}
T_{\cV^\#}^* \, T_\cV = \sum_{i\in \, \IN{m}}S_\cV\, ^{-1}\, V_i^*V_i=I_\cH 
\peso{and} 
S_{\cV^\#} = \sum_{i\in \, \IN{m}}S_{\cV}^{-1}\, V_i^*\, V_i \, S_{\cV}^{-1} = S_{\cV}^{-1}\ 
\ . 
\eeq
Next we generalize the notion of dual RS's\,: 
\EOE
\end{rem}

\begin{fed}\label{DV} \rm 
Let $\RSV$ and $\RSW \in \RS$.  
We   say that $\cW$ is a {\bf dual} RS 
for $\cV$ if $ T_\cW^*\, T_\cV = I_\cH\, $, 
or equivalently if 
$x  =\sum_{i\in \, \IN{m}} W_i^* \,V_i \,x  $ for every $x\in \H $.

\pausa
We denote the set of all dual RS's for a fixed $\cV \in \RS$ by  
$
\cD(\cV) \igdef \{\cW\in \RS: T_\cW^*\, T_\cV = I_\cH\,\}  \ .
$   
Observe that  $\cD(\cV) \neq \vacio$ since $\cV^\#\in \cD(\cV) $. 
\EOE
\end{fed}

\begin{rem} \label{MP sost}
Let $\cV\in \SP$. Then $\cV  \in \RS \iff T_\cV^*$ is surjective. 
In this case, a system  $\cW\in \cD(\cV)$ \sii  its synthesis 
operator $T_\cW^*$ is a pseudo-inverse of $T_\cV\,$. 
Indeed,  $\cW\in \cD(\cV) \iff T_\cW^*\, T_\cV = I_\cH\,$. Observe 
that the map $\RS \ni \cW \mapsto T_\cW^*$ is one to one. 
Thus,   in the context of RS's   each $(m,\k,d)$-RS has many duals that are $(m,\k,d)$-RS's. 
This is one of the advantages of the RS's setting.

\pausa
Moreover,  
the synthesis operator  $T_{\cV^\#}^*$ of the canonical dual $\cV^\#$ 
corresponds to the Moore-Penrose pseudo-inverse of $T_\cV\,$. Indeed, notice that 
 $T_{\cV}\,T_{\cV^\#}^* = T_{\cV}\, S_{\cV}\inv T_{\cV}^* 
 \in L(\cK)^+$, so that it  
 is an orthogonal projection. 
From this point of view, the canonical dual $\cV^\#$ has some optimal properties that come from the theory of pseudo-inverses.

\pausa
On the other hand the map $\SP \ni \cW \mapsto T_\cW^* \in \lkh$ is $\R$-linear. 
Then, for every $\cV\in \RS$,  the set $\cD(\cV)$ 
of dual systems is convex in $\SP$, because the 
set of pseudoinverses of $T_\cV$ is convex in $\lkh$. 
\EOE

\end{rem}

\section{Erasures and lower bounds.}
	
It is a known result in frame theory that, for a given frame $\cF=\{f_i\}_{i\in I}$, the set $\cF'=\{f_i\}_{i\in I, \, i\neq j}$ is either a frame or a incomplete set for $\hil$.
In \cite{[CasKu2]} P. Casazza and G. Kutyniok give examples where this situation does not occur in the fusion frame setting. Considering fusion frames as a particular case of  reconstruction systems we can rephrase their result in the following way: 

\begin{teo}[Casazza and Kutyniok]\label{teo CasKu}\rm
Let $\cV=\{V_i\}_{i \in \IM}\in \PRO$ with bounds $A_\cV \coma B_\cV$. 
If $\sum_{i\in J} \|V_i\|^2 <A_\cV\,$ then  the sequence 
$\cV_J\igdef \{V_i\}_{i \in \IM\setminus J}$ is a projective RS for $\hil \cong \C^d$ 
with bounds $A_{\cV_J} \ge A_\cV-\sum_{i\in J} \|V_i\|^2$ and $B_{\cV_J} \le B_\cV\,$. 
\end{teo}

\pausa
As they notice in \cite{[CasKu2]} with an example, this is not a necessary condition.
On the other side, in \cite{As}, M. G. Asgari  proves that, 
under certain conditions, a single element can be erased from 
the original fusion frame (in our setting, a projective RS), 
and he obtains different lower bounds for the resulting reconstruction system: 

\begin{teo}[Asgari]\label{teo As}\rm
Let $\cV=\{V_i\}_{i \in \IM}\in \PRO$ with bounds $A_\cV \coma B_\cV$. 
Suppose that there exists $j\in \IM$ such that
$M_j \igdef I_d - V_j^*V_j S_\cV^{-1} \in \matinvd$, 
then $\cV^j=\{V_i\}_{i \neq j}$ is a projective RS for 
$\hil \cong \C^d$ with bounds $A_{\cV_J} \ge \frac{A_\cV^2}{A_\cV+\|V_j\|^2\|M_j^{-1}\|^2}$ 
and $B_{\cV_J} \le B_\cV\,$.
\end{teo}

\pausa
Actually,  Asgari's result can be generalized to any subset $J$ of $\IM\,$ and general RS's. 
In the following statement we shall give 
necessary and sufficient conditions which guarantee 
that the erasure of  $\{V_i\}_{i\in J}$ of a non necessary projective 
$\RSV\in \RS$  provides another  RS. 
Recall that the sharp bounds for $\cV$ are given by 
$A_\cV = \|S_\cV\inv \|\inv$ and $B_\cV = \|S_\cV\|$. 

\begin{teo}\label{cotas nosiglias}
Let $\RSV \in \RS(m,\k,d)$  with bounds $A_\cV \coma B_\cV$. 
Fix a subset 
  $J\subset \IM$ and consider the matrix 
$M_J \igdef I_d-\sum_{i\in J} V_i^*V_i\,S_\cV^{-1} \in \mat$.   
Then,  
\beq\label{da RS}
\cV_J=(V_i)_{i\in \IM \setminus J}  \quad \mbox{ \rm is a RS for  \ \ $\hil \cong \C^d$ } 
\iff  M_J \in \matinvd  \ .
\eeq
In this case $S_{\cV_J} = M_J\, S_\cV$ and   $\cV_J$ has bounds 
$A_{\cV_J} \ge \frac{A_\cV}{\|M_J^{-1}\|}$ and $B_{\cV_J} \le B_\cV\,$.
\end{teo}

\begin{proof}
The equality $S_{\cV_J} = M_J\, S_\cV$ follows from the following fact: 
$$
M_J = I_d - \sum_{i\in J}V_i^*V_i\,S_\cV^{-1} = S_\cV \,S_\cV^{-1}
-\sum_{i\in J}V_i^*V_i\,S_\cV^{-1} = \sum_{i\notin J}V_i^*V_i\ S_\cV^{-1} = 
S_{\cV_J}\,S_\cV^{-1} \ .
$$ 
This 
implies the equivalence of Eq. \eqref{da RS}. 
On the other hand, 
\[
\frac{A_\cV}{\|M_J^{-1}\|}=\|S_\cV^{-1}\|^{-1} \, \|M_J^{-1}\|^{-1} \leq \|(M_J\, S_\cV)^{-1}\|^{-1}= \|S_{\cV_J}^{-1}\|^{-1}  = A_{\cV_J} \ .
\]
The fact that $0<S_{\cV_J} \le S_\cV$ assures that  $B_{\cV_J} \le B_\cV\,$.
\end{proof}

\pausa
In the case $J=\{j\}$, the lower bound in  Theorem \ref{cotas nosiglias} 
is greater than that obtained in \cite{As}: 

\begin{pro}\label{pro era}
Let $\cV$, and $M_J$ be as in Theorem \ref{cotas nosiglias}, with $J=\{j\}$. Then 
\beq\label{es mejor}
\frac{A_\cV^2}{A_\cV+\|V_j\|^2\|M_J^{-1}\|^2} \ \leq  \ \frac{A_\cV}{\|M_J^{-1}\|}\ . 
\eeq
\end{pro} 

\begin{proof}
We can suppose $\|M_J^{-1}\|\geq 1$, since otherwise \eqref{es mejor} is evident. 
Note that 
$$
\barr{rl}

\|V_j\|^2\geq A_\cV & \implies \ 
\frac{A_\cV^2}{A_\cV+\|V_j\|^2\|M_J^{-1}\|^2} \ \leq \ 
\frac{A_\cV^2}{\|V_j\|^2\|M_J^{-1}\|^2} \ \le \  \frac{A_\cV}{\|M_J^{-1}\|^2} 
 \ \leq  \ \frac{A_\cV}{\|M_J^{-1}\|}\ .
\earr
$$
But if $\|V_j\|^2<A_\cV\,$, then
$\|I_d-M_J\|=\|V_j^*V_j\, S_\cV^{-1}\|\leq \frac{\|V_j\|^2}{A_\cV}<1$. 
Therefore 
$$
\barr{rl}
 \|M_J^{-1}\|\leq \frac{A_\cV}{A_\cV-\|V_j\|^2} & \implies 
A_\cV \,\|M_J^{-1}\| \leq A_\cV +\|V_j\|^2 \|M_J^{-1}\|
\leq  A_\cV +\|V_j\|^2 \|M_J^{-1}\|^2 
\ ,
\earr
$$
which clearly  implies the inequality of Eq. \eqref{es mejor}.
\end{proof}

\begin{rem}
Let  $J\inc \IM\,$,  $\cV\in \RS$, and $M_J$ be as 
in Theorem \ref{cotas nosiglias}. 
Assume that $\|\sum_{i\in J} V_i^*\, V_i\| < A_\cV\,$   
(compare with the hypothesis $\sum_{i\in J} \|V_i\|^2<A_\cV\,$  of Theorem \ref{teo CasKu}). 
Then, as in the proof of Proposition \ref{pro era}, it can be shown that
under this assumption it holds that 
$\|I_d-M_J\|<1 \implies M_J \in \matinvd$ and 
that the lower bounds satisfy  
$$
\barr{rl}
A_\cV-\sum_{i\in J} \|V_i\|^2 & \leq  \ 
A_\cV-\|\sum_{i\in J} V_i^*\, V_i\| \le \frac{A_\cV}{\|M_J^{-1}\|} 
\le A_{\cV_J}\ .
\earr
$$ 
Hence Theorem \ref{cotas nosiglias} generalizes Theorem \ref{teo CasKu} to general RS's with better bounds.
The matrix $M_J$ also serves to compute the canonical 
dual system $(\cV_J)^\#$: If we denote 
$\cV^\#  = \{W_i \}\subim $ and 
$\cV^\# _J = \{W_i \} _{i\notin J}\,$, then the formula 
$S_{\cV_J} = M_J\, S_\cV$ of Theorem \ref{cotas nosiglias} 
gives the equality  
$$
\cV^\# _J  \cdot M_J\inv \igdef 
 \{W_i \, M_J\inv \} _{i\notin J}= 
 \{V_i\, S_\cV\inv  \, M_J\inv \} _{i\notin J}
 = \{V_i \, S_{\cV_J} \inv \} _{i\notin J} = (\cV_J)^\#   \ .
 $$
 That is, 
 $(\cV_J)^\# $ is the truncation  of  the canonical dual $\cV^\#$ modified with 
 $M_J\inv\,$.
\EOE
\end{rem}

\section{Geometric presentation of  $\RS$.}\label{RG}

In this section we shall study several objects related with 
the sets $\RS$ from a geometrical point of view. 
On one hand, this study 
is of independent interest. On the other hand, 
some geometrical results  of this section will be necessary 
in order to characterize the minimizers of the joint potential, 
a problem that we shall consider in Section \ref{sec joint}.

\subsection{General Reconstruction systems}
\begin{num}\label{top en SP}
Observe that we can use on $\SP$ the natural  metric 
$\|\cV\|_{_2} = \big(\,\sum_{i\in \IM} \|V_i\|_{_2}\big)\rai $ for 
$\RSV \in \SP$. Note that 
$$
\|\cV\|_{_2}^2 = \suml_{i\in \IM} \|V_i\|_{_2}^2 = \|T_\cV\|_{_2}^2 \peso{\big( \, in 
the space $\lhk 
\,$ \big) .} 
$$
With this metric it is easy to see that in $ \RS \inc \SP $ the following conditions hold: 
\ben
\item 
The space $\RS$ is open in $\SP$, since the map 
$\RSO:\SP \to \op$ given by $\RSO(\cV) = S_\cV = T_\cV ^*\,T_\cV$ 
(for $\cV \in \SP$\,) is continuous. 
\item On the other hand, if we fix $\cV \in \RS$, then the set 
$\cD(\cV)$ is {\bf closed} in $\SP$, because the map 
$\SP \ni \cW \mapsto T^*_\cW\, T_\cV \in \op$ is continuous. 
Observe that the equality $T^*_\cW\, T_\cV = I_\cH \implies T^*_\cW $ is surjective, so that 
$\cW\in \RS$. \EOE
\een
\end{num}

\begin{num}\label{diff RS}
Given a  {\bf surjective }  $A\in L( \cK\coma 
\cH\,)$, 
let us consider $P_{\cK_i} \, A^* \in L(\cH \coma \cK_i)$ \fori. 
Then $A$  produces a system $\cW_A = (P_{\cK_i} \, A^* )_{i\in \IM} \in \RS$ such that 
$$
T_{\cW_A}^* = A 
\ \mbox{ and } \ S_{\cW_A} = AA^* \in \glh^+\ .
$$ 
Therefore, given a fixed $\RSV\in \RS$,  we can parametrize 
$$
\RS = \{ U\cdot \cV \igdef (P_{\cK_i}\, U\, T_\cV )_{i\in \IM} : U \in \glk \} \ . 
$$
In other words, the Lie group $\glk$ acts 
transitively on $\RS$, where the action is given by 
the formula $U\cdot \cV = (P_{\cK_i}\, U\, T_\cV )_{i\in \IM} \,$. 
Indeed, for every $x = (x_i)_{i\in \IM} \in \cK$,  
\beq\label{TVU}
\barr{rl}
T_{U\cdot \cV}^*\, x & = \suml _{i\in \IM} T_\cV^* \,U^* \,P_{\cK_i} \, x  
= T_\cV^* \,U^* \, \suml _{i\in \IM} \,P_{\cK_i} \, x   = T_\cV^* \, U^* \, x \ . 
\earr
\eeq
Therefore $T_{U\cdot \cV}^* = T_\cV^* \, U^*  \in \lkh$, 
which  is {\bf surjective} for every $U \in \glk$, so that  $U\cdot \cV \in \RS$. 
Hence $T_{U \cdot \cV} = U \, T_\cV\,$, which shows that this is indeed an action. 
On the other hand, for every  $\cW \in \RS$, 
since both $T_\cW^*$ and $T_\cV^*$ are surjective, then there exists $U \in \glk$ such that 
$T_{\cW}^* = T_\cV^* \, U^*$. Therefore we have that  $ \cW = U\cdot \cV$. 

\pausa
Fix $\RSV\in \RS$. Then we can define a continuous  surjective map 
$$
\pi _\cV : \glk \to \RS \peso{given by} \pi _\cV (U) = U \cdot \cV \peso{for} U \in \glk \ .
$$
The isotropy subgroup of this action is $\cI_\cV = \pi _\cV \inv (\cV)= 
\{ U \in \glk : U\big|_{R(T_\cV)} = \mbox { Id } \}$.
Indeed, looking at Eq. \eqref{TVU} we see that $U\cdot \cV = \cV \iff T_\cV^*\,U^* = T_\cV^* 
\iff U \, T_\cV = T_\cV\,$. 
In \cite{CPS} it is proved that these facts are sufficient to assure
that $\RS$ is a smooth submanifold of $\SP$ (actually 
it is an open subset)  such that the map 
$\pi _\cV : \glk \to \RS$ becomes a smooth submersion. 
On the other hand, we can parametrize 
$\cD(\cV)$ in two different ways\,: 
\beq\label{pepe}
\barr{rl}
\cD(\cV) & 
= \Big\{ \cW \in \SP \  : \   T_\cW^* = T_{\cV^\#}^* + G \ , \ \  G\in \lkh \ \mbox{ and } \  G\big|_{R(T_\cV)} 
\equiv 0  \Big\} 
\\&\\ 
& = \Big\{ U \cdot \cV^\#   \ : \  U \in \glk  \ \mbox{ and } \  
P \, U ^* \, P  = P  \Big\} 
\ , \quad \mbox{where $P = P_{R(T_\cV)} \,$ .} \earr
\eeq
Indeed, just observe that 
$\ker  T_{\cV^\#}^*  = \ker S_V\inv T_\cV ^* = R(T_\cV)\orto = \ker P$. Therefore  
$$ 
T_{U \cdot \cV^\#}^* \  T_\cV = T_{\cV^\#}^* \, U^* \, T_\cV = I_\cH \iff 
U^*\, x \in x + \ker T_{\cV^\#}^* 
\ \ \mbox{for every $x\in R(P)$ ,} 
$$
which means exactly that $PU^*P = P$. \EOE
\end{num}

\begin{rem}\label{rem descomposicion} This geometric presentation is similar to the presentation of vector frames 
done in \cite{CPS}. The relationship is based on the following fact: 

\pausa
The space $\RS$ can be seen as an agrupation in packets of vector frames. 
Recall that  $n = \tr \,\k = \sum\subim k_i\,$. 
Let us denote by $\cE_i = \{ e_1^{(i)} \coma \dots \coma e_{k_i}^{(i)}\} $ the canonical ONB 
of each $\cK_i = \C^{k_i}$, and the set $\cE = \bigcup\subim \, \cE_i \,$, which is a 
reenumeration of the canonical ONB of the space $\cK = \bigoplus\limits_{i\in \, \IN{m}}  \ \cK_i \cong \C^n$. 
Then,  there is a natural one to one correspondence 
\beq\label{rep RS}
\RS \ni \RSV \longleftrightarrow 
\Big( \, ( \, V_i ^* \, e_j^{(i)} \, )_{j\in \IN{k_i}} \, \Big)\subim \ = 
\big(  \, T_\cV ^* \, e \, \big)_{e\in\, \cE} 
\ \in \ \cH^n \ ,
\eeq
where the right term is a general $n$-vector frame for $\cH$. On the other hand, fixed the ONB 
$\cE$ of $\cK$, the set of $n$-vector frames for $\cH$ can be also identified with the space 
$E(\cK \coma \cH) \igdef
\{ A \in \lkh : A $ is surjective $\}$, via the map 
$A \longleftrightarrow \big(  \, A \, e \, \big)_{e\in\, \cE} \,$.

\pausa
The geometrical representation of $\RS$ given before is the natural 
geometry of the space of epimorphisms 
$E(\cK \coma \cH)$ under the (right) action of $\glk$. Through all these 
identifications we get 
 the correspondence $\RS \ni \cV  \longleftrightarrow T_\cV^* \in E(\cK \coma \cH)$. 

\pausa
Observe that, in terms of Eq. \eqref{rep RS}, a system $\RSV\in \RS$ satisfy that
$\cV\in \PRO \iff $ each subsystem $( \, V_i ^* \, e_j^{(i)}\, )_{j\in \IN{k_i}}$  is a multiple of 
an orthonormal system in $\cH$.   \EOE
\end{rem}

\subsection{Projective RS's with fixed weights}
Given a fixed  sequence of weights $\v=(v_i)_{i\in \IM}\in \R_{>0}^m\,$, 
we define the set of projective RS's with fixed set of weights $\v$:  
\beq\label{PROv}
\PROv \igdef \Big\{ \, \RSV \in  \PRO  : \ 
\|V_i\|_{sp} =v_i   \, \text{ for every } i\in \IN{m} \Big\} \ .
\eeq
Denote by $\tau=\suml\subim v_i^2\,k_i\,$. 
Observe that $\tr \, S_\cV = \suml \subim \tr \, V_i^* \, V_i =  \tau $
for every $\cV \in \PROv$. In what follows we shall denote  by 
$$
\cM_d(\C)^+_\tau \igdef \{A\in \cM_d(\C)^+\, : \, \tr \,A=\tau \} \  \peso{ and }\
\gld_\tau \igdef \cM_d(\C)^+_\tau \cap \matinvd \, .$$ 
the set of $d\times d$ positive  and positive invertible operators with fixed trace $\tau$, 
endowed with the metric and geometric structure induced  by 
those of $\matinvd$.

\pausa
In this section we look for conditions which assure that the smooth  map
\beq\label{map RSO}
\RSO : \PROv \to \gld_\tau \peso{given by} \RSO(\cV ) = S_\cV = 
\sum \subim V_i ^*\, V_i \ , 
\eeq
for every $ \RSV \in \PROv $, has smooth local cross sections. 
Before giving these conditions and the proof of their sufficiency, we 
need some notations and two geometrical lemmas: 
Fix $d\in \N$. For every $k \in \IN{d}\,$, we denote by  
$\cI(k\coma d) = \{U\in  L(\C^k\coma\C^d) : U^*U = I_k\}$  
the set of isometries. 
Given an $m$-tuple $\k=(k_i)_{i\in \IM}\in \IN{d}^m \inc \N^m$,  
we denote by 
$$
\cI(\k\coma d) \igdef \bigoplus\subim  \cI(k_i\coma d)\ \inc 
\bigoplus\subim   L(\cK_i \coma\cH) \cong L(\cK\coma \cH) \ ,$$
 endowed with the product (differential, metric) structure
 (see \cite{AC} for a description of the geometrical structure). 
  Similarly,   
let $\text{Gr}(k,d)$ denote the Grassmann manifold of orthogonal projections of rank $k$ in $\C^d$ and let  
$$
\gr(\k\coma d) \igdef \bigoplus\subim  \gr (k_i\coma d) \inc \op^m \ ,
$$ 
with the product smooth structure (see \cite{CPR}).

\begin{lem}\label{fundamental}
Consider the smooth map  
$\Phi:\cI(\k\coma d)\rightarrow \gr(\k\coma d)$ 
given by 
$$
\Phi( \cW )=(W_1\, W_1^* \coma \dots \coma  W_m\,W_m^*) 
\peso{for every} \RSW 
\in \cI(\k\coma d)  \ .$$ 
Then  $\Phi$ has smooth local cross sections 
around any point $\cP = (P_i)\subim \in  \gr(\k\coma d)$ toward  every 
$\cW \in \cI(\k\coma d) 
$ such that $\Phi(\cW) = \cP$.  
In particular, $\Phi$ is open and surjective. 
\end{lem}
\proof Since both spaces have a product structure, it suffices 
to consider the case $m=1$.  
It is clear that the map  $\Phi$ is surjective. 

\pausa
For every $P \in \gr(k\coma d)$, 
the $C^\infty$ map $\pi_P : \cU(d) \to \gr(k\coma d)$ given by $\pi_P(U) = UPU^*$ 
for $U\in \cU(d)$ is a submersion with a smooth local cross  section (see \cite{CPR})
$$
h_P : U_P \igdef \{Q\in \gr(k\coma d) : \|Q-P\|<1 \} \to \cU(d)
\peso{such that} h_P(P) = I_d \ .
$$ 
For completeness we recall that,  for every $Q \in U_P\,$, the matrix  
$h_P(Q) $ is the unitary part in the polar decomposition  of 
the invertible matrix $QP + (I_d -Q) (I_d -P)$.  
Then, fixed $W\in \cI(k \coma d)$ such that 
$\Phi (W) = P$, we can  define the following smooth local cross section for $\Phi\,$: 
\beq
s_{P\coma W} : U_P \to \cI(k\coma d) 
\peso{given by} s_{P\coma W}(Q) = h_P(Q) \, W  \ , \peso{for every} Q\in U_P \ .
\QEDP
\eeq

\pausa
We shall need the following result from \cite{MRS}. In order to state it 
we recall the following notions and introduce some notations:
\ben
\item Fix $\v=(v_i)_{i\in \IM}\in\R_{>0}^m \, $. 
We shall consider the smooth map
\begin{equation}\label{el fi}
\Psi_\v  
: \cI(\k\coma d)
\rightarrow \cM_d(\C)^+\peso{given by} \Psi_\v(\cU)=\sum_{i\in \IM} v_i^2\ U_i\,U_i^*
\end{equation}
for every $\RSU \in \ \cI(\k\coma d)$. 
\item Given a set $\cP = \{P_j :j\in \IM\}\inc \cM_d(\C)^+$, 
we denote by 
\beq\label{conm}
\cP' = \{P_j :j\in \IM\}' = \{A\in \cM_d(\C): AP_j = P_j \, A \peso{for every} j \in \IM\}\, .
\eeq
Note that $\cP'$ is a (closed) unital selfadjoint subalgebra of $\cM_d(\C)$.
Therefore, 
\beq\label{hay proye}
\mbox{$\cP' \neq \C\, I_d \iff $ there exists 
a non-trivial orthogonal projection $Q \in \cP'$ .} 
\eeq
\een

\begin{lem}[\cite{MRS}] \label{citanosa}
Let $\v=(v_i)_{i\in \IM}\in\R_{>0}^m \, $ and $\cP=\{P_i\}_{i\in \IM}\in \gr(\k\coma d)$. Denote by $\tau=\sum_{i\in \IM}v_i^2\,k_i\,$. 
Then 
the  map $S_\v:\gr(\k\coma d)\rightarrow \cM_d(\C)_\tau ^+$ given by
\begin{equation}\label{el sv}
 S_\v(\cQ) 
= \sum_{i\in \IM} v_i^2\ Q_i \peso{for} \cQ= \{Q_i\}_{i\in \IM}\in \gr(\k\coma d)
\end{equation} 
is smooth and, if  $\cP$ satisfies that $\cP'=\CC\, I_d\,$, then 
\begin{enumerate}
\item The matrix  $S_\v (\cP) \in \gld_\tau\,$. 
\item The image of $S_\v$ contains an open neighborhood of $S_\v(\cP)$ in $\cM_d(\C)^+_\tau$.
\item Moreover, $S_\v$ has a 
smooth local cross section around $S_\v(\cP)$ towards $\cP$. 
\QED
\end{enumerate}
\end{lem}

\begin{num}\label{gamma diff}
The   set 
$\cI_0(\k \coma d) = \{\cW \in \cI(\k \coma d) : S_\v\circ \Phi(\cW) \in \gld\}$  is 
open in  $\cI(\k \coma d)$. Observe that its definition does not depend 
on  the sequence $\v=(v_i)_{i\in \IM}\in \R_{>0}^m\,$ of weights.  Moreover, 
the map $\gamma : \cI_0(\k \coma d) \to \PROv$ given by 
\beq\label{la gama}
\gamma (\cW ) = \{ v_i \, W_i^*\} \subim \in \PROv 
\peso{for every} \RSW \in \cI_0(\k \coma d) \ ,
\eeq
is a homeomorphism. Hence, using this map $\gamma$ 
we can endow $\PROv$ with the differential 
structure which makes $\gamma$ a diffeomorphism. With this structure, 
each space $\PROv$ becomes a submanifold of $\RS$. 
 It is in this  sense in which 
 the map $\RSO: \PROv \to \gld_\tau$  
defined in Eq. \eqref{map RSO} is smooth. Indeed, 
we have that 
\beq \label{tres cachos} 
\RSO = S_\v \circ \Phi \circ \gamma \inv  \ ,
\eeq
where $\Phi:\cI(\k\coma d)\rightarrow \gr(\k\coma d)$ is the smooth
map defined in Lemma \ref{fundamental}. 
Now we can give an answer to the problem posed in the beginning of 
this section. \EOE
\end{num}

\begin{fed}\rm 
Let $\v=(v_i)_{i\in \IM}\in \R_{>0}^m\,$ and $\RSV\in  \PROv(m,\k,d)$. 
We say that the system $\cV$ is {\bf irreducible} if
$C_\cV\igdef \{V_i^*V_i : i\in \IM\}'=\C\,I_d\,$. \EOE
\end{fed}

\pausa
In Section \ref{Exas} we show examples of reducible and irreducible systems. 
See also Remark \ref{redu}.

\begin{teo} \label{lema secciones}
Let $\v=(v_i)_{i\in \IM}\in \R_{>0}^m\,$ 
and $\tau = \suml\subim v_i ^2 \, k_i \, $. 
If we fix an {\bf irreducible} system $\cV\in \PROv(m,\k,d)$, then   
the map $\RSO: \PROv \to \gld_\tau$  defined in Eq. \eqref{map RSO} 
has a smooth local cross section  
around $S_\cV$ which sends $S_\cV$ to $\cV$. 
\end{teo}
\proof
We have to prove that 
there exists an open neighborhood $A$ of $S_\cV$ in $\gld_\tau$ and a smooth 
map $\rho : A \to \PROv$ such that $\RSO \,(\, \rho(S)\,) = S $ for every 
$S\in A$ and $\rho(S_\cV) = \cV$. 

\pausa
Denote by  $P_i = P_{R(V_i^*)}$ \fori \, , and consider the system 
$$
\gamma\inv(\cV) = \RSU \in \cI(\k\coma d) \peso{given by} U_i = v_i\inv \, V_i^*\in I(k_i \coma d) \ \ i \in \IM \  . 
$$
Observe that $\Phi(\cU) = \cP= \{P_i\}\subim \in  \gr(\k\coma d)\,$ and 
$ S_\v(\cP) = S_\cV\,$. By our hypothesis, we know that 
$\cP' = \{V_i^*V_i : i\in \IM\}'=\C\,I_d\,$. Let $\al : A \to \gr(\k\coma d)$ be the 
smooth section for the  map $S_\v:\gr(\k\coma d)\rightarrow \cM_d(\C)_\tau ^+$
 given by Lemma \ref{citanosa}. Hence $A$
is an open neighborhood of $S_\cV= S_\v(\cP)$ in $\gld_\tau\,$, and $\al (S_\cV) = \cP$. 

\pausa
Take now the cross section $\beta : B \to \cI(\k \coma d)$ for the map 
$\Phi:\cI(\k\coma d)\rightarrow \gr(\k\coma d)$  
given by Lemma \ref{fundamental}, such that  $B$ 
is an open neighborhood of $\cP$ in $\gr(\k\coma d)$, and that $\beta (\cP) = \cU$. 

\pausa
Finally we recall the diffeomorphism  
$\gamma : \cI_0(\k \coma d) \to \PROv$
defined in Eq. \eqref{la gama},  
where $\cI_0(\k \coma d) = \{\cW \in \cI(\k \coma d) : S_\v
\circ \Phi(\cW) \in \gld\}$ is an 
open subset of  $\cI(\k \coma d)$ such that $\cU\in \cI_0(\k \coma d)$. 
Note that $\gamma (\cU) = \cV$. 
Changing the first neighborhood $A$ by some smaller open set, 
we can define the announced smooth cross section for 
the map $\RSO$ by 
$$
\rho = \gamma \circ \beta \circ \al : A \inc \gld_\tau \to \PROv  \ .
$$ 
Following our previous steps, we see that $\rho(S_\cV) = \cV$ and that 
\beq
\RSO \stackrel{\eqref{tres cachos}}{=} 
S_\v \circ \Phi \circ \gamma \inv \ \implies 
\RSO (\rho (S)\,) = S \peso{for every} S\in  A  \ . 
\QEDP 
\eeq

\begin{rem}\label{las dist}\rm
In order to compute ``local" minimizers for different 
functions defined on $\RS$ or some of its subsets, we shall 
consider two different (pseudo) metrics: 
Given $\cV=\{V_i\}_{i\in \IM}$  and $\cW=\{W_i\}_{i\in \IM}\in \RS$,  
we recall the (punctual) metric defined in \ref{top en SP}: 
$$
d_P(\cV,\,\cW)=\left(\ \sum_{i\in \IM }\|V_i-W_i\|_{_2}^2\right)\rai =
\|T_\cV-T_\cW\|_{_2}=\|T_\cV^*-T_\cW^*\|_{_2}  \ . 
$$ 
We consider also a pseudo-metric 
defined by 
$ 
d_S(\cV,\,\cW)=\|S_\cV-S_\cW\| \ . 
$

\pausa 
Let $A\inc \RS$ and $f: A \to \R$ a continuous map. Fix  $\RSV\in A$. 
Since the map $\cV \mapsto S_\cV$ is continuous, 
it is easy to see that if $\cV$ is a local $d_S$ minimizer of $f$ over $A$, then  
$\cV$ is also a local $d_P$ minimizer. The converse needs not to be true. 

\pausa
Nevertheless, it is true under some assumptions: 
Theorem \ref{lema secciones} shows that if $\cV$ is a local $d_P$ minimizer 
of $f: \PROv \to \R$, 
in order to assure that $\cV$ is also a local $d_S$ minimizer it suffices
to assume that $ \{V_i^*V_i : i\in \IM\}' = \C\, I_d \,$, i.e. that 
$\cV$ is irreducible.  
\EOE
\end{rem}

\section{Spectral pictures}
Recall that $(\R_{+}^d)^\downarrow$ is the 
set of vectors $\mu \in \R_+^d$ with non negative and decreasing entries. If all
the entries are positive (i.e., if $\mu_d >0$), we write $\mu \in (\R_{>0}^d)^\downarrow$. 
Given $S\in \matrec{d}^+$, we write $\la(S) \in (\R_{+}^d)^\downarrow$ the 
decreasing vector of eigenvalues of $S$, counting multiplicities. 
We denote by $S^\dag$ 
 the Moore-Penrose pseudo-inverse of $S$.
We shall also use the following notations: 
\ben
\item Given $x\in \C^n$ then $D(x) \in \mat $ denotes the 
diagonal matrix with main diagonal $x$. 
\item If $d\le n$ and $y \in \C^d$, we write $ (y\coma 0_{n-d})\in \cene$, where  
$0_{n-d}$ is the zero vector of $\C^{n-d}$. 
In this case, we denote by $D_n(y) = D\big( \, (y\coma 0_{n-d})\,\big) \in \cene$. 
\een 
Given  $\cA\inc \matpos$ we consider 
its {\bf spectral picture}:
$$ 
\Lambda(\cA)=\{\lambda(A):\ A\in \cA\}\subseteq (\R_+^d )^\downarrow \ ,
$$ 
We say that  $\Lambda(\cA)$ {\bf determines} 
$\cA$ whenever $A\in \cA$ if and only if 
$\lambda(A)\in \Lambda(\cA)$. It is easy to see that 
this happens 
if and only if the set $\cA$ is 
saturated with respect to unitary equivalence.

\subsection{The set of dual RS's}
\begin{fed}\rm
Let $\cV\in \RS$. We denote by 
\begin{equation}\label{defi espec} 
\Lambda(\cD(\cV)\,)=\{\lambda(S_\cW):\ \cW\in \cD(\cV)\}
\inc  (\R_{>0}^d)^\downarrow \ ,
\end{equation} that is, the spectral picture of the set of 
all dual RS's for $\cV$. 
\EOE
\end{fed}

\pausa
The following result gives a characterization of $\Lambda(\cD(\cV))$.

\begin{teo} \label{TP}
Let $\RSV \in \RS$ and $\mu\in (\R_{>0}^d)^\downarrow$. 
We denote by $n= \tr \,\k $. 
Then the following conditions are equivalent: 
\ben
\item\label{cond1} 
The vector $\mu \in \Lambda(\cD(\cV))$. 
\item\label{cond2} There exists an orthogonal projection 
$P\in \matrec{n} $ such that $\rk \,P=d$ and  
\beq\label{item2}
\lambda\left( P\, D_n(\mu )\, P\right)= 
\big(\, \lambda(S_\cV^{-1}),0_{n -d}\,\big) = 
\la (G_\cV^\dag)  \ ,
\eeq
where $G_\cV = T_\cV\, T_\cV^* \in \matnpos$ is the Gram matrix of $\cV$. 
\een
\end{teo}

\proof
Let $\cW\in \cD(\cV)$ 
with $\lambda(S_\cW)=\mu$. 
Then $T_\cW^*\,T_\cV=I$ and 
\begin{equation}\label{ecua comp}
G_\cV \, G_\cW \, G_\cV  = T_\cV\,(T_\cV^*\, T_\cW)\,(T_\cW^*\,T_\cV)\,T_\cV^* = 
T_\cV\,T_\cV^* = G_\cV  \  
\implies \ 
Q \, G_\cW \,  Q = G_\cV^\dagger\  , 
\end{equation} 
where $Q = G_\cV \, G_\cV ^\dag = P_{R(T_\cV)}\,$. 
Note that $\rk \, Q = \rk \, T_\cV = d$, 
since $\cV $ is a  RS. 
Also
$$
\la(G_\cW) = \lambda(T_\cW\,T_\cW^*)=(\lambda(T_\cW^*\,T_\cW),0_{n -d})=
(\lambda(S_\cW),0_{n -d}) = (\mu \coma 0_{n -d}) \ . 
$$ 
Then there exists  $U\in \matu $ such that 
\begin{equation}\label{ecua uni}
U^* \, D(\mu \coma 0_{n -d})  \, U =  U^* \, D_n(\mu ) \, U = 
U^*\ D_n\big(\,  \lambda(S_\cW) \,\big) \ U=T_\cW\,T_\cW^*  = G_\cW \ .
\end{equation} 
Let $P = U\,Q\,U^*$. Note that $\rk \, P =\rk \, Q = d$.  
Using \eqref{ecua comp} and \eqref{ecua uni} we get the  item 2 :
$$
\lambda\left( P\, D_n(\mu )\, P\right)= 
\lambda( U\,Q\,U^* \, D_n(\mu ) \, U\,Q\,U^*)\stackrel{\eqref{ecua uni}}{=}
 \la(Q \, G_\cW \,  Q)\stackrel{\eqref{ecua comp}}{=}
\lambda( G_\cV^\dagger)=(\lambda(S_\cV^{-1}) \coma 0_{n -d})\ .$$ 
Conversely, assume  that there exists the projection $P\in \matposn$ of item 2. 
Observe that there always exists $\cU\in \RS$ such that 
$\la(S_\cU)= \la(T_\cU^*\,T_\cU)=\mu$. Then 
$$
\la( G_\cU ) = \la(T_\cU\,T_\cU^*)=(\mu \coma 0_{n -d}) \in (\R_{+}^n)^\downarrow \ .
$$  
Let  $V\in \matu $ such that $V^*\, G_\cU\, V=D_n(\mu)$. 
Denote by $Q= VPV^*$. Then 
we get that 
\begin{equation}\label{ecua algo} 
\lambda(Q\, G_\cU\, Q) = \la(P \, V^*\, G_\cU\, V \, P)  =
\lambda\left( P\, D_n(\mu )\, P\right) \stackrel{\eqref{item2}}{=} 
(\lambda(S_\cV^{-1}),0_{n -d}) = \la(G_\cV^\dag)\ .
\end{equation} 
Then 
there exists $W\in \matu $ such that 
$ W^*\,(Q\, G_\cU\, Q)\, W=G_\cV^\dagger\,$. 
Observe that 
$$
\rk \, Q = d \peso{and} W^*(R(Q)\,) \Inc R(G_\cV^\dagger) = R(G_\cV) =R(T_\cV) 
 \ \implies  \ W^*QW = P_{R(T_\cV)} \ .
$$ 
Moreover, $G_\cV \, G_\cV ^\dag = G_\cV ^\dag \, G_\cV  = P_{R(G_\cV)} = 
P_{R(T_\cV)} = W^*QW $. Then 
$$
\barr{rl}
G_\cV \ & = G_\cV \, G_\cV^\dag \, G_\cV = G_\cV \, (W^*\,Q\, G_\cU\, Q\, W) \, G_\cV
\\&\\
& = 
G_\cV \,P_{R(G_\cV)}\, (W^* \,G_\cU \,W)  \, P_{R(G_\cV)} \,  G_\cV  
= G_\cV \, (W^* \,G_\cU \,W)  \,  G_\cV  \ . \earr
$$
We can rewrite this fact as 
$T_\cV \big( \, T_\cV^* \, W^* T_\cU \, T_\cU^* \, W \, T_\cV \, \big) \, T_\cV^* = 
T_\cV \, T_\cV^* \,$. Since $T_\cV^*$ is surjective, 
\beq\label{casi dual}
(T_\cV^* \, W^* T_\cU) \, (T_\cU^* \, W \, T_\cV) = I_\cH \ \implies 
V_d = T_\cU^* \, W \, T_\cV \in \cU(d) \ .
\eeq
Finally, take $\cW = \{P_{\cK_i}\, W\, T_\cU \, V_d\}_{i\in \IM}  \in \SP$.  
Observe that 
$$
S_\cW =  \sum \subim V_d^* \, T_\cU^*\,W^*\, P_{\cK_i}\, W\,   T_\cU \, V_d =
 V_d^* \, T_\cU^*\, T_\cU \, V_d = V_d^* \, S_\cU \, V_d \in \gld \ .
$$
Then $\cW \in \RS$ and $\la(S_\cW) = \la(S_\cU) = \mu $. 
Similarly,  $T_{\cW} = W \, T_\cU \, V_d\,$. By  Eq. \eqref{casi dual}, 
we deduce that $T_\cW^* \, T_\cV = V_d^* \, T_\cU^* \, W \, T_\cV = V_d^*\, V_d = I_\cH\,$, 
so that $\cW\in \cD(\cV)$.  
\QED

\begin{rem} \label{FP}
Let $\cV\in \RS$ and $\mu\in (\R_{>0}^d)^\downarrow$ as in Theorem  \ref{TP}. 
It turns out that condition \eqref{item2} can be characterized in terms of 
interlacing inequalities. 

\pausa
More explicitly, let us denote by 
\bce
$\gamma = \mu^\uparrow \in (\R_{>0}^d)^\uparrow$ \ ,  \ \ \ so that  \ \ \ 
$\gamma_i = \mu_{d-i+1}$  \ \ for every  \ \ $i \in \IN{d}\ $. 
\ece
Similarly, we denote by 
$\rho = \la(S_\cV^{-1})^\uparrow  = (\la_i(S_\cV)\inv)\subim \in (\R_{>0}^d)^\uparrow$.  K. Fan and G. Pall showed that the existence of a projection $P$ satisfying \eqref{item2} 
is equivalent to the following inequalities: 
\begin{enumerate}
\item $\mu_{d-i+1} = 
\gamma_i\geq \rho_i= \lambda_{i}(S_\cV)^{-1}$ for every  $i \in \IN{d}\,$. 
\item If $n = \tr \, \k < 2\, d$ and we denote $r = 2\, d-n \in \N$, then  
$$
\gamma_i  \ \le \ \rho _{i+n-d}   = \lambda_{i+n-d}(S_\cV)^{-1} 
= \lambda_{2\,d  -n-i+1}(S_\cV^{-1}) 
\peso{if} 1\le i \le r 
\ .
$$
\end{enumerate}
This fact together with Theorem \ref{TP} give a complete description 
of the spectral picture of the RS operators $S_\cW$ for every $\cW\in \cD(\cV)$, which we 
write as follows. 
\EOE
\end{rem}

\begin{cor}Let $\RSV \in \RS$, $n = \tr \, \k$ and fix $\mu \in  (\R_{>0}^d)^\downarrow $. Then, the set $\Lambda(\cD(\cV))$ can be characterized as follows: 
\ben 
\item If $n\ge 2\, d$, we have that 
\beq\label{mayor 2d}
\mu  \in \Lambda(\cD(\cV)) \iff 
\mu_j \ge \la_j (S_\cV\inv) = \lambda_{d-j+1}(S_\cV)\inv  \peso{for every} j \in \IN{d} \ .
\eeq
\item If $n<2\,d$, then $\mu \in  \Lambda(\cD(\cV)) \iff  \mu $ satisfies \eqref{mayor 2d} and
also the following conditions: 
\beq\label{menor 2d}
\mu_i^\uparrow = \mu_{d-i+1}  \ \le \  \lambda_{i+n-d}(S_\cV)^{-1} 
= \lambda_{2\,d  -n-i+1}(S_\cV^{-1}) 
\peso{for every}  i \le  2\, d-n  \ .
\eeq
\een
\end{cor}
\proof 
It is a direct consequence of Theorem \ref{TP} and the 
Fan-Pall inequalities described in Remark \ref{FP}. \QED

\begin{cor}\label{La conv} 
Let $\cV \in \RS$.  Then $ \Lambda(\cD(\cV))$ is a convex set.
\end{cor}
\proof It is clear that the inequalities given in Eqs. \eqref{mayor 2d} and \eqref{menor 2d}
are preserved by convex combinations. Observe that also the set $(\R_{>0}^d)^\downarrow $
is convex. \QED

\begin{cor}\label{unico FFP}
Let $\RSV \in \RS$. 
If $\cW\in \cD(\cV)$ then 
\begin{equation}\label{cot inf}
\FP(\cW) \igdef \tr \, S_\cW^2  \geq \tr \, S_\cV^{-2}=\sum_{i=1}^d\lambda(S_\cV)_i^{-2}  = 
\FP(\cV^\#)\ .
\end{equation} 
Moreover,  $\cV^\#$ is the \rm unique \it 
element of $\cD(\cV)$ which attains the lower bound in \eqref{cot inf}. 
\end{cor}
\begin{proof}
The inequality given in Eq. \eqref{cot inf} is a direct consequence of \eqref{mayor 2d}. 
With respect to the uniqueness of $\cV^\#$, fix another  $\cW\in \cD(\cV)$. Then 
 the equalities $T_\cW^*\,  T_\cV = T_{\cV^\#}^*\,  T_\cV= I$ imply that 
$T_\cW^* = T_{\cV^\#}^* + A$, for some $A\in \lkh$ that
satisfies $ A\,  T_\cV =0$. 
With respect to $\cV^\#$, note that 
 $R(T_{\cV^\#}) =  R(T_\cV \, S_\cV\inv ) =  R(T_\cV) \inc \ker A$, so that 
also $A \, T_{\cV^\#} = 0$. Thus,
\beq\label{una mas}
\barr{rl}
\tr \, S_\cW & = \|T_{\cV^\#}^* + A\|_{_2}^2 = 
\tr \,\big(  T_{\cV^\#}^* T_{\cV^\#} \,\big) + \tr \,\big(  A  A^* \,\big)
+ 2\, \Preal \, \tr \,\big( A \, T_{\cV^\#} \,\big) 
 \\& \\ & = 
\tr \, S_{\cV^\#} + \|A\|_{_2}^2 
\ . 
\earr
\eeq
On the other hand, if the lower bound in Eq. \eqref{cot inf} 
is attained $\cW$, using \eqref{mayor 2d} we can deduce that 
$\la  (S_\cW) = \la(S_{\cV^\#})$. Then also 
$\tr \, S_\cW = \tr \, S_{\cV^\#}\,$. But the previous 
equality forces that in this case  
$A=0$ and hence $\cW=\cV^\#$.
\end{proof}

\subsection{RS operators of projective systems}
In this section we shall fix the parameters $(m,\k,d)$ and 
the sequence $\v=(v_i)_{i\in \IM}\in \R_{>0}^m\,$ of weights. 
Now we give some new notations: 
First, recall that the set of projective RS's with fixed set of weights $\v$ is  
$$
\PROv = \PROv(m,\k,d) = \Big\{ \, \{V_i\}\subim \in  \PRO  : \ 
\|V_i\|_{sp} =v_i   \, \text{ for every } i\in \IN{m} \Big\} \ .
$$
We consider the set of  operators $S_\cV$ for  $\cV\in \PROv$ 
and its spectral picture:
\begin{equation}\label{el Delta} 
\OPv \igdef  \{S_\cV:\ \cV\in \PROv\} \peso{and} 
\LOPv \igdef\{\la(S): S\in \OPv\}  \inc  (\R_{>0}^d)^\downarrow  \ .
\end{equation}
We shall give a characterization of the set $\LOPv$ in terms of the 
Horn-Klyachko's theory of sums of hermitian matrices. In order 
to do this we shall describe briefly the basic facts about the spectral
characterization obtained by Klyachko \cite{Klya} and Fulton \cite{Ful0}. 
Let 
$$
\mathcal K_d^r=\big\{(j_1,\ldots,j_r)\in (\IN{d}) ^{\, r} :\   j_1<j_2\ldots<j_r \big\} \, . 
$$
For $J=(j_1,\ldots,j_r)\in \mathcal K_d^r\,$, define the
associated partition $\lambda(J)=(j_r-r,\ldots,j_1-1) \,$.
For $r\in \IN{d-1}$ denote by $LR_d^{\,r}(m)$ the set of
$(m+1)$-tuples $(J_0,\ldots,J_m)\in (\mathcal K_d^r)^{m+1}$, such
that the Littlewood-Richardson coefficient of the associated
partitions $\lambda(J_0),\ldots,\lambda(J_m)$ is positive, i.e. one
can generate the Young diagram of $\lambda(J_0)$ from those of
$\lambda(J_1),\ldots,$ $\lambda(J_m)$ according to the
Littlewood-Richardson rule (see \cite{Ful0}).

\pausa
The theorem of Klyachko gives a characterization of the spectral
picture of the set of all sums 
of $m$  matrices in $\cH(d)$ with fixed given spectra, in terms on a series of inequalities 
involving the $(m+1)$-tuples in $LR_d^{\,r}(m)$ (see  
 \cite{Klya} for a detailed formulation). We give a description of this result
in the particular case where these $m$ matrices are multiples of projections: 

\pausa
\begin{lem}\label{EL K} \rm
Fix the parameters $(m, \k , d)$ and 
$\v\in  \R_{>0}^m\,$ 
$\mu \in (\R_{+}^m)^\downarrow\,$. Then there exists a sequence 
$\{P_i\}_{i\in \IM} \in \gr(\k\coma d)$
such that $\mu = \la \left( \ \sum\subim v_i ^2 \, P_i\right) $ \sii  
\begin{equation}\label{la ec posta}
\barr{rl}
 \tr \, \mu & = \ \sum \subim \, v_i^2 \, k_i   \peso{and} \sum_{i\in J_0}\mu_i \, \leq \, 
\sum_{i \in \, \IM}\,  v_i^2\, |\,J_i\cap  \IN{k_i} \,| \ , 
\earr
\end{equation}
for every $r\in \IN{d-1}$  and every $(m+1)$-tuple $(J_0,\ldots,J_m) \in 
LR_d^{\,r}(m)$. \QED
\end{lem}

\begin{pro}\label{con suf y nec}
Fix the parameters $(m, \k , d)$ and the vector  
 $\v\in  \R_{>0}^m\,$ of weights.  
Fix also a positive matrix  $S\in \gld$.
Then, 
$$
S\in \OPv \iff \la(S) \in \LOPv \iff 
\la(S)  \ \ \mbox{satisfies Eq. \eqref{la ec posta} .}
$$
\end{pro}
 
\proof  The set $\OPv \inc \gld$ 
is saturated by unitary equivalence. Indeed, if 
$\cV\in \PROv$ and $U \in \cU(d)$, then $\cV\cdot U \igdef \{V_i\,U\}_{i\in \IM}\in \PROv$ 
and $U^*S_\cV U = S_{\cV\cdot U}\in \OPv\,$. 
This shows the first equivalence. 
On the other hand, using Lemma \ref{fundamental} 
and Eq. \eqref{la gama}, we can assure that  
an ordered vector $\mu\in \LOPv$ if and only if 
$\mu_d>0$ and there exists a sequence of projections 
  $\cP=\{P_i\}_{i\in \IM} \in \gr(\k\coma d)$
such that $\mu = \lambda(S_\v(\cP)\, ) =  
\la \left( \ \sum\subim v_i ^2 \, P_i\right) $. 
Hence, the second equivalence follows from Lemma \ref{EL K}. 
\qed

\begin{cor}\label{comp conv}
For every set $(m, \k , d)$ of parameters and every vector  
 $\v\in  \R_{>0}^m\,$ of weights, 
\ben
\item The set  $\LOPv$  is convex. 
\item Its closure $\ov{\LOPv}$ is compact.
\item A vector $\mu \in \ov{\LOPv} \, \setminus \LOPv 
\iff \mu_d = 0$. In other words, 
\beq
\ov{\LOPv} \ \cap \ \R_{>0}^m = \LOPv \ .
\eeq
\een
\end{cor}
\begin{proof}
Denote by $\eme $ the set of vectors $\la\in (\R_{+}^d)^\downarrow$ 
which satisfies  Eq. \eqref{la ec posta}. 
It is clear that $\eme $ is compact and convex. 
But Proposition \ref{con suf y nec} assures that 
$\LOPv = \eme \cap \R_{>0}^d \inc \eme\,$. 
This proves items 2 and 3. Item 1 follows 
by the fact that also $\R_{>0}^d $ is convex. 
\end{proof}

\begin{rem}
With the notations of Corollary \ref{comp conv}, 
actually $\ov{\LOPv} = \eme$. This fact is not obvious 
from the inequalities of Eq. \eqref{la ec posta}, but can be deduced 
using  Lemma \ref{EL K}. 
Indeed, it is clear that if $\cP\in \gr(\k\coma d)$ and $S_\v(\cP)\notin \gld$, 
then  $S_\v(\cP)$ can be approximated by matrices 
$S_\v(\cQ)$ for  sequences $\cQ\in \gr(\k\coma d)$ such that
$S_\v(\cQ)>0$. Using Lemma \ref{fundamental} and 
Eq. \eqref{la gama},  this means that these matrices 
 $S_\v(\cQ)\in \OPv\,$. \EOE
\end{rem}

\section{Joint potential of projective RS's}\label{sec joint}
Fix the parameters $(m,\k,d)$. We consider the set of  dual pairs 
associated to $\PROv$: 
$$
\DRSv = \DRSv(m,\k,d) \igdef 
\Big\{ \, (\cV,\,\cW) \in  \PROv \times \RS \ : \ 
\cW\in \cD(\cV) \, \Big\} \ .
$$
We consider on $\DRSv$ the joint potential: Given $(\cV \coma \cW)\in \DRSv$, 
let 
\begin{equation}\label{defi pot conj}
\FP(\cV,\cW) \igdef\FP(\cV)+\FP(\cW)=\tr \, S_\cV^2 +\tr \, S_\cW^2  \in \R_{>0} \ .
\end{equation}
We shall describe the structure of the minimizers of the joint potential both from a spectral and a geometrical point of view. We will denote by 
\begin{equation}\label{el p} 
p_\v = p_\v(m,\k,d)  \igdef \inf_{} \, \{\, \FP(\cV,\cW) \, : \, 
(\cV,\,\cW)\in \DRSv \,\} \ .
\end{equation}

\begin{pro}\label{pro pares otimos}
For every set $(m, \k , d)$ of parameters, the following properties hold: 
\begin{enumerate}
\item The infimum $p_\v $ in Eq. \eqref{el p} is actually a {\bf minimum}. 
\item Let $\tau=\sum\subim v_i^2\,k_i\,$. For every pair $(\cV,\,\cW)\in \DRSv$ we have that 
\beq\label{pv tight}
\FP(\cV,\,\cW )\ge p_\v \geq \frac{\tau^4 +d^4}{d\, \tau^2}  \ , 
\eeq
\item This lower bound is attained if and only if $\cV$ is tight 
($S_\cV  = \frac{\tau }{d} \,I_d$) and $\cW = \frac{d}{\tau}\, \cV = \cV^\#$.
\een
\end{pro}
\proof
Given $(\cV,\,\cW)\in \DRSv$, Corollary \ref{unico FFP} asserts that 
 $\FP(\cV,\,\cV^\#)\le \FP(\cV,\,\cW)$ and also that equality holds only if $\cW=\cV^\#$. 
 Thus 
 \begin{equation}\label{ecua simp}
 p_\v=\inf_{\cV\in \PROv}\FP(\cV,\cV^\#)
 \stackrel{\eqref{SMP}}{=}\inf_{\cV\in \PROv} \ 
 \sum_{i=1}^d\lambda_i(S_\cV)^2+ \lambda_i(S_\cV)^{-2} \ .
 \end{equation}
Consider the {\bf strongly convex} map $F:\R_{>0}^d\rightarrow \R_{>0}$ given by $F(x)=\sum_{i=1}^d x_i^2+x_i^{-2}$, 
for $x\in \R_{>0}^d\,$.  
Observe that $\FP(\cV,\cV^\#) = F(\la(S_\cV)\,)$ for every $\cV \in \PROv$. 
By Corollary \ref{comp conv} we know that 
 $\LOPv$ is  convex subset of $(\R_{>0}^d)^\downarrow\,$, 
 and it becomes also compact under a restriction of the type 
 $\la_d \ge \eps$ (for any $\eps>0$). 
 Since a strongly convex function defined in a compact convex set 
 attains its local (and therefore global) 
 minima at a unique point, it follows that there exists a unique 
 $\la_\v=\la_\v(m,\k,d)\in \LOPv$ such that 
\beq\label{el la top}
F(\,\la_\v\,) =\min_{\lambda\in \Lambda_\v (m,\k,d)} F(\lambda) 
= p_\v\ . 
\eeq 
This proves item 1. 
Moreover, using Lagrange multipliers it is easy to see that  
the restriction of $F$ to the set 
$(\RR_{>0}^d)_\tau:=\{\x \in \RR_{>0}^d \,: \, \tr(\x)=\tau \, \}$ reaches its minimum in $\x=\frac{\tau}{d}\cdot \uno$.
Since $\LOPv\subset (\RR_{>0}^d)_\tau$ we get that
$$ 
\FP(\cV,\,\cV^\#)=F(\lambda(S_\cV)\,)\ge F(\frac{\tau}{d}\cdot \uno )
=\frac{\tau^4 +d^4}{d\, \tau^2} \peso{for every} 
\cV \in \PROv \ ,
$$ 
and this lower bound is attained if and only if $\lambda(S_\cV)=\frac{\tau}{d}\cdot \uno_d\, $. 
Note that in this case  $S_\cV=\frac{\tau}{d} \, I_d\,$, 
and therefore  $\cV^\#=\frac{d}{\tau}\, \cV$. 
\QED

\pausa
 Recall that we use in $\RS$  the metric
$
d_P(\cV,\,\cW)=(\ \suml_{i\in \IM }\|V_i-W_i\|_{_2}^2\,)\rai 
=\|T_\cV^*-T_\cW^*\|_{_2} $ and the pseudometric 
$d_S(\cV,\,\cW)=\|S_\cV-S_\cW\|$ 
for pairs 
 $\cV=\{V_i\}_{i\in \IM}$  and $\cW=\{W_i\}_{i\in \IM}\in \RS$.  

\begin{lem}\label{dp MP}
If a pair $(\cV,\,\cW)\in \DRSv$ 
is local $d_P$-minimizer of the joint potential in $\DRSv$, then $\cW = \cV^\#$. 
\end{lem} 
\proof
We have shown in Eq. \eqref{pepe}
that, since $\cW\in \cD(\cV)$, then $T_\cW^* = T_{\cV^\#}^* + A$, for some $A\in \lkh$ such that
$ A\,  T_\cV =A\, T_{\cV^\#} = 0\in \op$. 
Recall from Remark \ref{MP sost} that the set $\cD(\cV)$ is convex. 
Then the line segment
$\cW_t = t \cW  + (1-t)\cV^\#\in \cD(\cV)$ satisfies that 
$T_{\cW_t}^* = T_{\cV^\#}^* + t A$ for every $t\in [0,1]$. 
Then, as in Eq. \eqref{una mas},   $S_{\cW_t} = 
S_{\cV^\#} + t^2 \, AA^*$ and 
$$
K(t) \igdef \FP (\cV\coma \cW_t ) = \FP (\cV\coma \cV^\# )+ t^4   \tr \, (AA^*)^2 + 
2 \, t^2 \tr \, T_{\cV^\#}AA^*T_{\cV^\#}^* \ ,
$$
for every $t\in [0,1]$. Observe that $K(1) = \FP (\cV\coma \cW ) $. 
 But taking one derivative of $K$, one gets that
if $A\neq0$ then $K$ is strictly increasing near $t=1$, which contradicts the 
local $d_P$-minimality for $(\cV\coma \cW)$. Therefore 
$T_{\cW_t}^* = T_{\cV^\#}^*$ and $\cW = \cV^\#$.
\QED

\begin{teo}\label{teo pares otimos} 
For every set $(m, \k , d)$ of parameters there exists  $\la_\v
= \la_\v(m,\k,d)\in (\R_{>0}^d)^\downarrow$ 
such that the following conditions are equivalent for pair $(\cV,\,\cW)\in \DRSv$: 
\ben
\item  $(\cV,\,\cW)$ is local $d_S$-minimizer of the joint potential in $\DRSv$. 
\item  $(\cV,\,\cW)$ is global minimizer of the joint potential in $\DRSv$. 
\item It holds that 
$\lambda(S_{\cV})=\la_\v$ and $\cW=\cV^\#$.
\een
\end{teo}

\begin{proof}
Take the vector $\la_\v$ defined in Eq. \eqref{el la top}. 
In the proof of Proposition \ref{pro pares otimos} we have already seen that 
a pair $(\cV,\,\cW)\in \DRSv$ is a global minimizer for $\FP \iff 
\cW= \cV^\#$ and $\la (S_\cV) = \la_\v\,$. This means that 
$2 \iff 3$.  

\pausa
Suppose now that  $(\cV,\,\cW)\in \DRSv$ is a local $d_S$-minimizer. By Remark \ref{las dist} 
we know that it is also a local $d_P$-minimizer and by Lemma \ref{dp MP} we have that 
$\cW=\cV^\#$. 
In this case, 
denote $\lambda=\lambda(S_\cV)$ and take  $U\in \cU(d)$ such that $U^*D_\lambda U=S_\cV\,$. 
Consider  the segment line 
 
$$
h(t)=t\,\la_\v+(1-t)\,\lambda \peso{for every} t\in [0,1] \ .
$$ 
Then $h(t)\in\LOPv$ 
for every $t\in [0,1]$,  since $\LOPv$ is a convex set 
(Corollary \ref{comp conv}). 
Consider the continuous curve $S_t = U^*D_{h(t)} U$ 
in $\OPv$ and a (not necessarily continuous) curve 
$\cV_t\in \PROv$ such that 
$S_0 = S_\cV\,$,  $\cV_0 = \cV$ and $S_{\cV_t} = S_t\,$ 
for every $t\in [0,1]$. 
Nevertheless, since the curve $S_t$ is continuous, we can assure that 
the map $t\mapsto \cV_t $ is $d_S$-continuous. 

\pausa
Finally, we can consider the map  $G: [0,1] \to \R $ given by 
$$
G(t) = \FP(\cV_t \coma \cV_t ^\#) = \tr \ S_{t }^2 + \tr \, S_{t }^{-2} 
= \sum_{i=1}^d  h_i(t)^2 + h_i(t)^{-2} =  
F(h(t)\,) 
$$ 
for $t\in [0,1]$, where $F$ is the map defined after Eq. \eqref{ecua simp}. 
Observe that $G(0) = \FP (\cV\coma \cV^\#)$ and $G(1) = p_\v\,$, 
by Eq. \eqref{el la top}. 
Then $G$  has local minima at $t=0$ and $t=1$. 
By computing the second derivative of $G$ in terms of the Hessian of $F$, we 
deduce that $G$ must be  constant, because otherwise it would be 
strictly convex. 
From this fact we can see that  the map $h$ is also constant, so that $\la_\v= \la$. 
Therefore $(\cV,\,\cW)= (\cV\coma \cV^\#)$ 
is a global minimizer. 
\end{proof}

\pausa 
Recall that  a system $\RSV\in  \PROv$  
is {\bf irreducible} if 
$C_\cV= \{V_i^*V_i : i\in \IM\}'=\C\,I_d\,$.

\begin{lem}\label{teo pares otimos tight} 
Fix the set $(m, \k , d)$ of parameters and the weights 
$\v=(v_i)_{i\in \IM}\in \R_{>0}^m\,$.  Assume that  
$\cV \in \PROv$ is irreducible. 
Then the following conditions are equivalent: 
\ben
\item The pair $(\cV \coma \cV^\#)$ is local $d_P$-minimizer of the joint potential in $\DRSv$. 
\item The pair  $(\cV \coma \cV^\#)$ is global minimizer of the joint potential in $\DRSv$. 
\item The system $\cV$ is tight, i.e. 
$S_{\cV} = \, \frac{\tau}{d}\, I_d \, $.
\een
Therefore in this case the vector $\la_\v$ of 
Theorem \ref {teo pares otimos} is 
$\la_\v =  \, \frac{\tau}{d}\, \uno_d\,$. 
\end{lem}

\proof 
Since $C_\cV = \C\, I_d\,$, we can apply 
Theorem \ref{lema secciones}. 
Then the map $\RSO: \PROv \to \gld_\tau$  defined in Eq. \eqref{map RSO} 
has a smooth local cross section  
around $S_\cV$ which sends $S_\cV$ to $\cV$.
Assume that there exists no $\sigma\in \R_{>0}$ such that 
$S_\cV = \sigma\,I_d\,$. In this case there 
exist $\alpha,\,\beta\in \sigma(S_\cV)$ such that $\beta>\alpha>0$.
Consider the map $g:[0,\frac{\beta-\alpha}{2}]\rightarrow \R_{>0}$ given by 
$$
 g(t)
=(\alpha+t)^2+(\alpha+t)^{-2} + (\beta-t)^2+(\beta-t)^{-2}\ .
$$ 
Then  $g'(0)=2(\alpha-\beta)-2(\frac{1}{\beta}-\frac{1}{\alpha})<0$,   
which shows that we can construct 
a continuous curve 
$M:[\,0 \coma \eps \,]\to \matinvd^+_\tau $ 
such that 
$M(0) = S_\cV \,  $ and 
$$
\barr{rl}
 \tr \,M(t)^2 +\tr\, M(t)^{-2} & 
 < \ \tr \, S_\cV ^2 + \tr \, S_\cV ^{-2} = \FP(\cV \coma \cV^\#)
 \peso{for every}  
 t\in(0 \coma  \eps\, ] \ .
 \earr
$$ 
Hence, using the continuous local cross section mentioned before, we can construct 
a $d_P$-continuous curve $\eme : [0\coma \delta \,] \to \PROv$ 
such that $\RSO \circ \eme = M$, $\eme(0)= \cV$ and  
$$
\barr{rl} 
\FP (\eme(t) \coma \eme(t)^\#) = \tr \, M(t) ^2+\tr M(t) ^{-2} &  <  \ 
\FP(\cV \coma \cV^\#)
\peso{for}  t\in( 0\coma \delta \,]  \ . 
\earr
$$
This shows that $(\cV \coma \cV^\#)$ is not a 
local $d_P$-minimizer of the joint potential in $\DRSv$. 
We have proved that $1\implies 3$. Note that
$3\implies 2$ follows from \eqref{pv tight} and $2\implies 1$ is trivial. 
\QED

\begin{rem} \label{redu}
It is easy to see that, if the parameters $(m\coma \k\coma d)$ allow the existence 
of at least one irreducible projective RS, then the set of irreducible systems 
becomes open and dense in $\PROv(m\coma \k\coma d)$. Nevertheless, it is not 
usual that the minimizers are irreducible, even if they are tight 
(see Remark \ref {cor resumen} and Examples \ref{E1} and \ref{E2}).  

\pausa
On the other hand, if the system $\cV\in \PROv$ 
is reducible, there exists a system 
$\cQ= \{Q_j\}_{j\in \IN{p}}$ of minimal projections of 
the unital $C^*$-algebra $C_\cV\,$ (with $p>1$). This means that 
\bit 
\item Each $Q_j \in C_\cV\,$, and $Q_j^2 = Q_j^* = Q_j\,$. 
\item $\cQ$ is a system of projections: $Q_j \, Q_k = 0 $ if $j\neq k$  
and  $\sum_{j\in \IN{p}} Q_j = I_\cH\,$. 
\item Minimality: The algebra $C_\cV\,$ has no proper sub projection of any $Q_j\,$. 
\eit 
By compressing the system $\cV$ to each subspace $\cH_j = R(Q_j)$
in the obvious way, it can be shown that every $\cV\in \PROv$ 
is an ``orthogonal sum" of irreducible subsystems. 
\end{rem}
\pausa
Another system of projections associated with $\cV$ are the 
spectral projections of $S_\cV\,$: If  
$\sigma (S_\cV) = \{\sigma_1\, , \dots, \,\sigma _r\}$, 
we denote these projections by 
$$
P_{\sigma_j} = 
P_{\sigma_j}(S_\cV) \igdef P_{\ker\, (S-\sigma_j \, I_d)}\in \matpos \ ,  \peso{for} j \in \IN{r} \ .
$$ 
Recall that $S_\cV \,  P_{\sigma_j} = \sigma _j \, P_{\sigma_j}\,$ and 
$\sum_{j=1}^r  \, P_{\sigma_j} = I_d\,$, 
so that $S_\cV = \sum_{j=1}^r \sigma _j \, P_{\sigma_j}\, $. 
\EOE

\begin{teo}\label{sum irr}
Fix $\v=(v_i)_{i\in \IM} \in \R_{>0}^m\,$. 
Let $(\cV,\,\cW)\in \DRSv$ be a $d_P$-local minimizer of the joint potential in $\DRSv$ with $\cV=\{V_i\}_{i\in \IM}\,$. Then 
\ben
\item The RS operator $S_\cV\in C_\cV\,= \{V_i^*V_i:\ i\in \IM\}'$. 
\item 
If 
$\sigma(S_\cV)=\{\sigma_1,\ldots,\sigma_r\}$, then  also 
$
P_{\sigma_i} = P_{\sigma_i}(S_\cV)\in C_{\cV} $ 
for every  $i\in \IN{r} \,$.
\een
\end{teo}
\proof
Recall that $\cV\in \PROv \inc \RS$ and hence $0\notin \sigma(S_\cV)$. 
On the other hand, we have already seen in 
Lemma \ref{dp MP} that $\cW$ must be $\cV^\#$. 
Let 
$\cQ=\{Q_j\}_{j\in \IN{p}}$ be a 
system of minimal projections of the unital $C^*$-algebra $C_\cV\,$, 
as in Remark \ref{redu}.

\pausa
Fix $j\in \IN{p}\,$ and denote by $\ese_j = R(Q_j)$.   
For every $i\in \IM$ put 
$\cT_i =V_i(\ese_j) \inc \cK_i\,$, $t_i = \dim \ete_i$  and 
$W_i=V_iQ_j \in L(\H_j \coma \cT_i) \,$. 
Since $Q_j\in C_\cV$ then each matrix $v_i \inv \, W_i^*$ is  an isometry, so that 
the compression of $\cV$ given by 
$\RSW \in \PROv(m\coma \t \coma s_j)$, where $\t = (t_1 \coma \dots \coma t_m)$ and 
$s_j = \dim \ese_j\,$. 
Recall that $S_\cV$ commutes with $Q_j\, $. This implies that 
 $\cW^\# $ is the same type of compression to $\RSv(m\coma \t \coma s_j)$ 
 of the system $\cV^\#$. 

\pausa
A straightforward computation shows that the pair 
$(\cW,\,\cW^\#)\in \DRSv(m\coma \t \coma s_j)$ is still
 a $d_P$-local minimizer of the joint potential 
in $\DRSv(m\coma \t \coma s_j)$. 
Indeed, the key argument is that one can ``complete" other systems in 
$\PROv(m\coma \t \coma s_j)$ near $\cW$ (and acting in $\ese_j$) 
with the fixed orthogonal 
complement $\{V_i(I_d -Q_j)\}\subim \,$, getting systems in $\PROv (m\coma \k\coma d)$ near $\cV$.  
It is easy to see that all the computations involved in the joint potential 
work independently on each orthogonal subsystem.  This shows the minimality of $(\cW,\,\cW^\#)$.

\pausa
Observe that $W_i^*W_i = Q_j V_i^*V_i Q_j = V_i^*V_i Q_j $ \fori . 
Therefore, the minimality of $Q_j$ in $C_\cV$ shows that 
the system $\cW$ satisfies that $C_\cW = \C \, I_{\ese_j}\,$.  
Hence, we can apply Lemma  \ref{teo pares otimos tight}  on $\ese_j\,$, and get 
that $S_\cW = \alpha_j\, I_{\ese_j}\,$ for some $\al_j>0$. But when we return to $\op$, we get that 
$S_\cV\,Q_j = \sum\subim  V_i^*V_i Q_j 
= \sum\subim  W_i^*W_i = S_\cW =\alpha_j\,Q_j\,$. In particular,  $\al_j \in \sigma (S_\cV)$.

\pausa
We have proved  that for every $j\in \IN{p}$ 
there exists $\al_j \in \sigma(S_\cV)$ such that 
$S_\cV\,Q_j=\alpha_j\,Q_j$  and hence each projector  $Q_j\leq P_{\alpha_j}=P_{\alpha_j}(S_\cV) \,$. 
Using  that $\sum_{j\in \IN{p}}Q_j=I_d$ we see that each 
\beq\label{part irr}
P_{\sigma_k} =\sum_{j\in J_k}Q_j\in C_\cV\ ,  \peso{where}
J_k=\{j\in \IN{p}:\ \alpha_j=\sigma_k\} \ . 
\eeq
Therefore also $S_\cV = \sum_{k\in \IN{r}} \sigma _k \, P_{\sigma_k}
\in C_\cV\, $.
\QED

\begin{rem}\label{cor resumen} \rm 
Theorem \ref{sum irr} assures that if $(\cV\coma \cV^\#)$ is a $d_P$-local 
minimizer of the joint potential in $\DRSv\,$, then  $\cV$ is 
an orthogonal sum 
of {\bf tight} systems in the following sense: 

\pausa
If $\sigma(S_\cV)=\{\sigma_1,\ldots,\sigma_r\}$, and we denote 
$\cH_j = R(P_{\sigma_j})= 
\ker\, (S-\sigma_j \, I_d)\,$ for every  $j\in \IN{r}\,$, 
then  $\cH =  \bigoplus_{j\in \IN{r}} \, \cH_j\,$. 
By Theorem \ref{sum irr} each $P_{\sigma_j} \in C_\cV\,$. 
Then, putting $d_j = \dim \cH_j\,$, 
\bce
 $\cK_{i \coma j} =V_i(\cH_j) \inc \cK_i\,$ \coma  $k_{i \coma j} = \dim \cK_{i \coma j}$  
 \ \ and  \ \ $\k^j = (k_{1 \coma j} \coma \dots \coma k_{m \coma j})$  \ ,
\ece
\fori \  and  $j\in \IN{r}\,$, 
we can define  the the  tight compression of $\cV$ to each $\cH_j\, $: 
$$
\cV^j = \{V_i\, P_{\sigma_j}\}_{i\in \IN{m}} \in \PROv (m\coma \k^j \coma d_j )
\peso{for} j \in \IN{r} \ .
$$ 
Indeed, since $P_{\sigma_j} \in C_\cV$ then $\cV^j$ is projective. Also 
$S_{\cV^j} = S_\cV \, P_{\sigma_j} = \sigma_j \, P_{\sigma_j}\,$, 
which means that $\cV^j$ is $\sigma_j$ - tight. 
Observe that the decomposition of each $\cV^j$ into irreducible tight systems 
(as in Remark \ref{redu}) follows from  the orthogonal decomposition  of $\cH_j$ given in  Eq. \eqref {part irr}.

\pausa
In particular, every $\cV\in \PROv$ such that $\la(S_\cV) =\la_\v $ 
(the unique vector of Theorem  \ref{teo pares otimos}) must 
have this structure, because in this case $(\cV\coma \cV^\#)$ is 
a $d_S$ (hence also $d_P$) local minimizer  of the joint potential in $\DRSv\,$. 
Observe that the structure of all 
global minimizers $\cV$ is almost the same: Since 
$\la(S_\cV)= \la_\v\,$, the number $r$ of tight components, 
the sizes $d_j$ and 
the tight constants $\sigma_j$ for each space $\H_j$ coincide 
for every such minimizer $\cV$. 

\pausa
Note that, if such a 
$\cV$ is not tight, then it can not be irreducible. On the other hand,  
its dual $\cV^\# $ can only be projective if 
$V_i \, P_{\sigma_j} = 0$ or $V_i$ for every 
$i \in \IM$ and $j \in \IN{r}\,$.  \EOE
\end{rem}

\section{Examples and conclusions}\label{Exas}

\pausa
The following two examples are about irreducible systems. 
\begin{exa}\label{E1} 
Let $d = k_1 + k_2\, $ and $\k = (k_1\coma k_2)$. Assume that $k_1 >k_2\,$. 
We shall see that, in this case, 
there is no irreducible (Riesz) systems in $\cP\cR\cS (2\coma \k \coma d)$. 
Observe that the situation is the same whatever the weights $(v_1\coma v_2)$ are.

\pausa
Indeed, if  $\cV = (V_1\coma V_2) \in \cP\cR\cS_\uno (2\coma \k \coma d)$, 
let  $\ese_i = R(V_i^*)$ and $P_i = P_{\ese_i}= V_i^*V_i$ for $i = 1,2$.  Then 
$\C^d = \ese_1 \oplus \ese_2\,$ (not necessarily orthogonal). 
Observe that $\dim S_1 =\dim S_2\orto = k_1\,$ and $2\, k_1 >d$. Hence 
$\ete= \ese_1 \cap \ese_2\orto \neq \trivial $. Since $P = P_\ete \le P_1$ 
and $P\le I_d- P_2\,$, then 
$P \in C_\cV$ and  $0 \neq P \neq I_d\,$. Therefore 
 $C_\cV  \neq \C\, I_d\,$.

\pausa
In particular, if the decomposition $\C^d = \ese_1 \oplus \ese_2\,$ 
is orthogonal, then $S_\cV = P_1 + P_2 = I_d\,$. So,  in this case $\cV$ 
is tight and reducible. 
\EOE
\end{exa}

\begin{exa}\label{E2} 
If $m \ge d$ and $\k = \uno_m\,$, then $\PRO(m\coma \k \coma d)$ is the set of 
$m$-vector frames for the space 
$\C^d$.  In this case $\RSF \in \PRO$ is reducible $\iff$ there exists
$J\inc \IM$ such that 
$\vacio \neq J \neq \IM$ and the subspaces $\gen\{f_i : i \in J\}$ and 
$\gen\{f_j : j \notin J\}$ are orthogonal. 

\pausa
Indeed, if $A = A^*$, then $A \in C_\cF\,  \iff\, $ every $f_i$ is an eigenvector of $A$. 
But different eigenvalues of $A$ must have orthogonal subspaces of eigenvectors. Observe that 
in this case the set of irreducible systems is an open and dense subset of $\PROv$, 
since it is the intersection of $2^m-2$ open dense sets (one for each fixed nontrivial 
$J\inc \IM$). 
\EOE
\end{exa}

\begin{num}\label{con mayo} {\bf Minimizers and majorization:} 
Theorem \ref{teo pares otimos} states that there exists a vector 
$\la_\v = \la_\v(m,\k,d)\in (\R_{>0}^d)^\downarrow$  such that a system 
$\cV\in \PROv (m,\k,d)$ satisfies that 
$(\cV,\,\cV^\#)$ is a global minimizer of the joint potential in $\DRSv$
\sii $\lambda(S_{\cV})=\la_\v\,$. 
This vector is found as the unique minimizer of the map 
$F(\la)=\sum_{i=1}^d \la_i^2+\la_i^{-2}$ on the convex set $\Lambda(\OPv)\,$.

\pausa
In all the examples where  $\la_\v$ could be explicitly computed, it 
satisfied a stronger condition, in terms of majorization 
(see \cite [Cap. II]{[Bat]} for definitions and basic properties).   
We shall see that in these examples there is a vector $\la \in \Lambda (\OPv\,)$ such that 
\beq\label{lav may}
\mbox{$\la \prec \la(S_\cV) $  \ for every  \ $\cV\in \PROv\,$  \ (the symbol 
\ $\prec$  \ means majorization) .}
\eeq 
Observe that such a vector $\la \in \Lambda (\OPv\,)$ 
must be the unique  minimizer  for  $F$ on $\Lambda (\OPv\,)$, since the map $F$ is  
permutation invariant and  convex. Hence $\la =\la_\v \,$. 
Moreover, those cases where $\la_\v $ satisfies Eq. \eqref {lav may}
have some interesting properties regarding the structure 
of minimizers of the joint potential. For example, 
that $\la_{\,t\v}(m,\k,d)=\,t^2 \la_{\v}(m,\k,d)$ for $t>0$, 
a fact that is not evident at all from the properties of these vectors.
\EOE
\end{num}
\begin{conj} \label{la conj}
For every set of parameters $(m,\k,d)$ and $\v\in \R_{>0}^d\,$, 
the vector  $\la_\v(m,\k,d)$ of Theorem \ref{teo pares otimos}  satisfies the
 majorization minimality of Eq. \eqref {lav may} 
on $\Lambda(\OPv)\,$. \EOE
\end{conj}

\begin{exa}\label{ejem para vectores} 
Given $\v=\v^\downarrow\in \R_{>0}^m\,$ and $d\le m$,  
the 
 $d$-irregularity of $\v$ is the 
index 
$$
\barr{rl}
r &= \ r_d(\v) \igdef \max\,  \big\{\, j\in \IN{d-1} \,:\ (d-j)\, v_j^2
\,> \, \sum_{i=j+1}^m v_i^2 \ \big\} \ , 
\earr
$$ 
or $r=0$ if this set is empty. 
In \cite[Prop. 2.3]{MR} (see also \cite[Prop. 4.5]{Illi}) 
it is shown that for any set of parameters 
$(m\coma\uno_m\coma d)$ and every $\v=\v^\downarrow \in \R_{>0}^m\,$,
there is $c\in \R$ such that
$$
\la_\v(m,d) \igdef (v_1^2 \coma \ldots \coma v_r^2 \coma c \, \uno_{d-r} )\in \Lambda(\OPv(m,\uno_m \coma d)\,)
$$ 
and it satisfies  Eq. \eqref{lav may}. Therefore 
$\la_\v(m,d)=\la_\v(m,\uno_m\coma d)$ by  \ref{con mayo}.
Thus, in the case of vector frames,  Conjecture \ref{la conj} is known to be true. 
 \EOE
\end{exa}

\pausa
In the following examples we shall compute explicitly the 
the vector $\la_\v$ and the global minimizers of the joint potential 
in $\PROv\,$. 
Since  we shall use Eq. \eqref {lav may} as our main tool (showing Conjecture 
\ref{la conj} in these cases),  we need a technical 
result about majorization, similar to  
 \cite[Lemma 2.2]{P}. 
Recall that the symbol $\prec_w$ means 
weak majorization.  
\begin{lem}\label{lem may} \rm
Let $\alpha \coma \gamma \in \R^n$, 
$\beta\in \R^m$ and $b\in \RR$ 
such that $b \le \min_{k \in \In} \gamma_k\,$.
Then, if 
$$
\tr\, (\gamma \coma b\,\uno_m) \, \le \, \tr\,(\alpha \coma \beta) 
\peso{and} \gamma \, \prec_w \, \alpha \ 
\implies \ (\gamma \coma b\,\uno_m) \, \prec_w \, (\alpha \coma \beta) \ . 
$$
Observe that we are not assuming that $(\alpha \coma \beta) = 
(\alpha \coma \beta)^\downarrow$. 
\end{lem}
\proof Let $h = \tr\, \beta $ and 
$\rho = \frac hm \, \uno_m\,$. 
Then it is easy to see that 
$$
\barr{rl}
\sum_{i\in \IN{k}} (\gamma^\downarrow \coma b\,\uno_m)_i &\le \ 
\sum_{i\in \IN{k}} (\alpha^\downarrow \coma \rho)_i  \ \le \  
\sum_{i\in \IN{k}}(\alpha^\downarrow \coma \beta^\downarrow) _i 
\peso{for every} k \in \IN{n+m}\ .
\earr
$$ 
Since  $(\gamma^\downarrow \coma b\,\uno_m) = 
(\gamma \coma b\,\uno_m)^\downarrow$, 
 we can conclude that 
  $(\gamma \coma b\,\uno_m)  \prec_w  (\alpha \coma \beta)$. 
 \QED
\begin{exa}\label{E m R} 
Assume that $\tr \k = d$. Then the elements of 
$\PRO_\v(m\coma \k \coma d)$ are {\bf Riesz} systems. 
Assume that the weights are ordered in such a way that $\v = \v^\downarrow$. We shall see that  
the vector 
$\la = (v_1^2\,\uno_{k_1}\coma \dots\coma v_m ^2\,\uno_{k_m})\prec \lambda (S_\cV)$ 
for every $\cV\in \PRO_\v(m\coma \k \coma d)$. Hence $\la $ satisfies 
 Eq. \eqref{lav may}, and $\la_\v = \la$ by  \ref{con mayo}.

\pausa
Indeed, given $\RSV\in \PRO_\v\,$, consider the projections 
$P_i = v_i^{-2}\, V_i^*V_i\,$  
and denote by $\ese_i = R(P_i)$ \fori. 
Then $S_\cV = \sum \subim v_i^2\, P_i\,$ and $\C^d = 
\bigoplus _{i \in \IM} \ese_i\,$ where the direct sum is not 
necessarily orthogonal. 
Let 
\bce 
$\ese = \bigoplus _{i \in \IN{m-1}} \ese_i\inc \C^d\ $ \ ,  \ \ 
$P = P_\ese$  \ \ and  \ \ $Q = I_d -P =  P_{\ese\orto}\,$ .
\ece
Consider the restriction 
 $A = \suml _{i =1}^{m-1} 
 \, v_i^2 \, P_i \in L(\ese)^+$. 
It is well known that the pinching matrix
$$
M = 
P \, S_\cV \, P + Q\, S_\cV \,Q = 
\bm {cc}  A  + v_m^2\, PP_2P &0\\0& v_m^2\, QP_m Q\em 
\barr{ll} \ese\\\ese\orto \earr 
$$
satisfies that $\la(M) \prec \la (S_\cV)$. 
Using an inductive argument on $m$ 
(the case $m=1$ is trivial),  
for the Riesz system $\cV_0= \{V_i\big|_\ese\}_{i \in \IN{m-1}}$ (for $\ese$) 
such that $S_{\cV_0} = A$, we can assure that 
$$
\gamma = (v_1^2\,\uno_{k_1}\coma \dots\coma v_{m-1} ^2\,\uno_{k_{m-1}})\prec \lambda (A)
\prec_w \la \big( \, A  + v_m^2\, PP_2P \,\big) =
\al  \peso{in} \R^{d-k_m} \ .
$$ 
Since $v_m \le v_{m-1}\,$,   Lemma \ref{lem may}  
assures that  
$\la = (\gamma \coma v_m^2\, \uno_{k_m})  
\prec (\al \coma \beta) = \la(M)$, 
where $\beta = \la (v_m^2\, QP_m Q) \in \R^{k_m}$. Hence, we have proved that  $\la  \prec \la(S_\cV)$. 

\pausa
Recall a system $\cV\in \PROv $ is a minimizer 
\sii $ \la (S_\cV) = \la_\v= \la   $. Now, it is easy to see that  
$ \la (S_\cV) = \la_\v   $  \sii  the 
projections $P_i $ are mutually orthogonal. 
\EOE
\end{exa}

\begin{exa} Assume that  the parameters $(m\coma \k \coma d)$ satisfy that 
$$
\mbox{$m =2$  \ \ and  \ \ $\tr \k = k_1 + k_2 >d$ \ ,  \ \ 
but  \ \ $k_1\neq d \neq k_2\ $.} 
$$
Fix $\v= (v_1\coma v_2)$  with $v_1 \ge v_2$. For the space 
$\PROv(2\coma \k \coma d)$ the vector $\la_\v$ of 
Theorem \ref{teo pares otimos} and all the global minimizers 
of the joint potential can be computed: 
Denote by 
$$r_0 = k_1 + k_2 -d \quad , \quad r_1 = k_1 - r_0\peso{and}
r_2 = k_2 -r_0 \ .
$$ 
We shall see that  the vector $\mu= (\,(v_1^2 +v_2 ^2)\, \uno_{r_0} 
\coma v_1 ^2\, \uno_{r_1}\coma v_2 ^2\, \uno_{r_2}) $ satisfies 
Eq. \eqref{lav may}, so that $\la_\v(2\coma \k \coma d) = \mu$ 
by \ref{con mayo}. Moreover, 
the minimizers are those  systems $\cV = (V_1 \coma V_2) \in \PRO_\v$ 
such that the two projections $P_i = v_i^{-2}\,V_i^*V_i$ (for $i = 1,2$) commute. 

\pausa
Indeed, if $\ese_i = R(P_i) = R(V_i^*) $ for $i = 1,2$, then $\eme _0 = \ese_1\cap \ese_2$ 
has $\dim \eme_0 = r_0\, $. Also $\eme_i = \ese_i \ominus \eme_0$ have $\dim \eme_i = r_i$ 
for $i = 1,2$. Hence 
$\C^d = \eme_0 \perp (\eme_1 \oplus  \eme_2)$ and 
$$
S_\cV = v_1^2\, P_1 + v_1^2\, P_2\, = \,(v_1^2 +v_2 ^2)\, 
P_{\eme_0} + v_1^2\, P_{\eme_1} + v_1^2\, P_{\eme_2} \ .
$$ 
Note that $\eme_1 \perp  \eme_2 \iff  
 P_1\,P_2 = P_2 \,P_1 = P_{\eme_0}\,$. In this case  
$\la(S_\cV) = \mu$. Otherwise, still 
$S_\cV\big|_{\eme_0} = 
(v_1^2 +v_2 ^2) \, I_{\eme_0} $ and $S_\cV(\eme_1 \oplus  \eme_2) =
\eme_1 \oplus  \eme_2 \, $. Hence, if 
we denote by $T = S_\cV\big|_{\eme_1 \oplus  \eme_2} = 
(v_1^2\, P_{\eme_1} + v_1^2\, P_{\eme_2})\big|_{\eme_1 \oplus  \eme_2} 
\in \mathcal{G}\textit{l}(\eme_1 \oplus  \eme_2)^+
\,$, then 
$\|T\|_{sp} \le v_1^2 +v_2 ^2$ and 
$$
S_\cV = \bm{cc} (v_1^2 +v_2 ^2)\, I_{r_0} & 0 \\ 0& T\em   \barr{ll}\eme_0 \\
\eme_0\orto\earr  \peso{with}
\la(S_\cV) = (\,(v_1^2 +v_2 ^2)\, \uno_{r_0} \coma \la (T)\, ) \in (\R_{>0}^d)^\downarrow \ . 
$$
Using Example \ref{E m R} for the space $\eme_1 \oplus  \eme_2\,$, 
we can deduce that  $(v_1^2\, \uno_{r_1}\coma v_2^2\, \uno_{r_2}) 
\prec \la(T)$. Therefore 
also $\mu = (\,(v_1^2 +v_2 ^2)\, \uno_{r_0} 
\coma v_1 ^2\, \uno_{r_1}\coma v_2 ^2\, \uno_{r_2})\prec 
(\,(v_1^2 +v_2 ^2)\, \uno_{r_0}   \coma \la (T)\, ) = \la(S_\cV)$. 
\EOE
\end{exa}

\begin{exa} Let $m=3$, $d=4$, $\k = (3\coma 2\coma 2) $ and $\v= \uno_3\,$. 
Denote by $\cE = \{e_i : i\in \IN{4}\}$ the canonical basis of $\C^4$. Then 
$\la_\uno(3\coma \k \coma 4) = (2\coma 2\coma \frac 32\coma \frac 32\,)$ 
and a minimizer is given by any system $\cV = 
\{V_i\}_{i\in \, \IN{3}} \in \PRO_\uno$ such that 
the subspaces $\ese_i = R(V_i^*)$ for $i \in \IN{3}$ are  
$$
\ese_1 = \gen\{e_1\coma e_2\coma e_3\} \quad , \quad  
\ese_2 = \gen \big\{\, e_1\coma w_2\, \big\}
\peso{and} 
\ese_3 = \gen\big\{\, e_2\coma w_3 \, \big\} \ , 
$$
where $w_2 = \frac{-e_3}{2} + \frac{\sqrt{3}\, e_4}{2}$ and 
$w_3= \frac{-e_3}{2} - \frac{\sqrt{3}\, e_4}{2}\ $. 
The fact that 
$\la(\ese_\cV) = (2\coma 2\coma \frac 32\coma \frac 32\,)$
for such a system $\cV$  is a direct computation. 
On the other hand, 
if $\cW= \{W_i\}_{i\in \, \IN{3}}\in \PRO_\uno(3\coma \k \coma 4)\,$, then there exist
unit vectors $x_2 \in R(W_1^*)\cap R(W_2^*)$ and $x_3 \in R(W_1^*)\cap R(W_3^*)$. 

\pausa
Denote by $\ete = \gen \{x_2\coma x_3\}$. If $\dim \ete = 1$ then 
$\la_1(S_\cW) \ge \api S_\cW\, x_2\coma x_2\cpi = 3$ and 
$\la_2(S_\cW) \ge 1$.  If $\dim \ete = 2$,  
using that $\ete\inc R(W_1^*)$ and 
 $x_i \in R(W_i^*)$ for $i=2,\,3$, we get 
$$
\barr{rl}
\la_1(S_\cW) + \la_2(S_\cW) &\ge \ 
 \sum_{i\in \, \IN{3}} \tr \big( P_\ete \, W_i^*W_i \, P_\ete\big)  \ge 
 \tr P_\ete + \tr P_{\gen\{x_2\}} + \tr P_{\gen\{x_3\}} = 4 \ . 
 \earr
$$
In any case, we have shown that  $(2\coma 2)\prec_w 
\al = (\la_1(S_\cW) \coma \la_2(S_\cW)\,)$. 
Therefore, using Lemma \ref{lem may}
we get that $(2\coma 2\coma \frac 32\coma \frac 32\,) \prec 
\la(S_\cW)\,$. Now,  apply  \ref{con mayo}. 

\pausa
The minimizers $\cV\in \PROv$ such that 
$\la(S_\cV) = (2\coma 2\coma \frac 32\coma \frac 32\,)$
have some interestig properties. For example they are the sum of 
two tight systems, $\cV^\#$ is not projective, and the 
involved projections do not commute. More precisely, the cosine of 
the Friedrich angles of their images 
are $c(\ese_i \coma \ese_j ) = \frac 12 $
for every $i \neq j$. 
\EOE
\end{exa}

\fontsize {9}{9}\selectfont

\noi {\bf Pedro Massey, Mariano Ruiz and Demetrio Stojanoff}

\noi Depto. de Matem\'atica, FCE-UNLP,  La Plata, Argentina
and IAM-CONICET  

\noi e-mail: massey@mate.unlp.edu.ar , maruiz@mate.unlp.edu.ar,  demetrio@mate.unlp.edu.ar

\end{document}